\documentclass[12pt,oneside,reqno]{amsart}
\usepackage[utf8]{inputenc}
\usepackage{graphicx}
\usepackage{amsthm}
\usepackage{amssymb}
\usepackage{amsbsy}
\usepackage[T1]{fontenc}
\usepackage{amscd}
\usepackage{verbatim}

\advance\hoffset by -0.5 truecm

\title{multipeak solutions for the Yamabe equation}

\date{}

\theoremstyle{plain}
\newtheorem{thm}{Theorem}[section]
\newtheorem{prop}[thm]{Proposition}
\newtheorem{lem}[thm]{Lemma}

\theoremstyle{definition}

\theoremstyle{remark}

\newtheorem*{dem}{Proof}
\newtheorem*{demTeo}{Proof of Theorem $\ref{main_thm}$}

\newtheorem*{demProp3}{Proof of Proposition $\ref{prop3}$}

\newcommand{\al}{{\alpha}}

\newcommand{\ep}{{\epsilon}}
\newcommand{\io}{{\iota}}

\def\V{\mbox{Var}}

\def\R\re

\def\V{\bf V}

\def \re{{\mathbb R}}

\def \C{{\mathbb C}}

\def \0{\lambda_{0}}

\def \Sp{{\mathbb S}}

\begin{document}

\author{Carolina A. Rey}
 \address{Departamento de Matem\'atica\\
 	 Universidad de Buenos Aires \\
Ciudad Universitaria, Pabell\'on I, (C1428EGA)\\
 Buenos Aires, Argentina.}
\email{carey@dm.uba.ar}

\author[J. M. Ruiz]{Juan Miguel Ruiz}
 \address{ENES UNAM \\
           37684 \\
          Le\'on. Gto. \\
          M\'exico.}
\email{mruiz@enes.unam.mx}

\begin{abstract}
Let $(M,g)$ be a closed Riemannian manifold of dimension $n\geq 3$ and $x_0 \in M$ be an isolated local minimum of the scalar curvature $s_g$ of $g$.
For any positive integer $k$ we prove that for $\epsilon >0$ small  enough the subcritical Yamabe equation 
$-\epsilon^2 \Delta u  +(1+ c_{N} \  \epsilon^2   s_g  ) u
= u^q$ has a positive $k$-peaks solution which concentrate around $x_0$, assuming that a constant $\beta$ is non-zero. In the equation $c_N = \frac{N-2}{4(N-1)}$ for an integer $N>n$ and
$q= \frac{N+2}{N-2}$. The constant $\beta$ depends on $n$ and $N$, and can be easily computed numerically, being 
negative in all cases considered. This provides solutions to the Yamabe equation on  Riemannian products $(M\times X , g+ \epsilon^2 h )$, where $(X,h)$ is a Riemannian manifold with constant positive scalar curvature.  
We also prove that solutions with small energy only have one local maximum.

\end{abstract}

\maketitle

\section{Introduction}

Consider a closed Riemannian manifold $(M,g)$ of dimension $n\geq 3$. The Yamabe problem consists of finding metrics of constant scalar curvature in
the conformal class of $g$, $[g]$. If we denote by $s_g$ the scalar curvature of $g$ and call $a_n = \frac{4(n-1)}{n-2}$, $p_n = \frac{2n}{n-2}$ 
(the critical Sobolev exponent) then for a positive function $u: M \rightarrow \re$ the metric $u^{p_n -2} g \in [g]$ has constant scalar curvature $\lambda \in \re$
if and only if $u$ is a solution to the {\it Yamabe equation}

\begin{equation}
-a_n \Delta u + s_g u = \lambda u^{p_n -1}
\end{equation}

A fundamental result proved by H. Yamabe \cite{Yamabe}, N. Trudinger \cite{Trudinger2}, T. Aubin \cite{Aubin} and R. Schoen \cite{Schoen}
says that there is always one solution which minimizes energy (which means that
the corresponding metric minimizes the total scalar curvature functional in its conformal class). But in general the solution is not unique and many
interesting results have been proved about multiplicity of solutions (see for instance references \cite{Brendle}, \cite{BP1}, \cite{BP2},
\cite{Piccione}, \cite{Henry}, \cite{Petean},
\cite{Pollack}, \cite{Schoen2}). 

Let $(M^n,g)$ be any closed Riemannian manifold and $(X^m,h)$ a Riemannian manifold of constant positive scalar curvature $s_h$. Let $N=n+m$.
We will be interested in positive solutions of the Yamabe equation for the
product manifold  $(M\times X, g+\epsilon^2 h)$:

\begin{equation}
\label{originalYamabe}
-a_{N} (\Delta_g + \Delta_{\epsilon^2 h})u+(s_g+\epsilon^{-2}s_h)u=u^{p_{N}-1}.
\end{equation}

\noindent The conformal metric $u^{p_{N}-2} ( g + \epsilon^2 h)$ then has constant scalar curvature. This case of Riemannian products has
recently been studied by several authors (see for instance \cite{Piccione}, \cite{Henry}, \cite{Petean2} and the references in them).

We restrict our study to functions that depend only on the first factor, $u:M\rightarrow \re$. We normalize  $h$ so that $s_h=a_{N}$, and let $c_N = a_N^{-1}$. Then  $u$ solves the Yamabe equation  if and only if (after renormalizing)
\begin{equation}
\label{Yamabe}
-\epsilon^{2} \Delta_gu+\left( c_N s_g \epsilon^{2}+1\right)u= u^{p_{N}-1}.
\end{equation}

We will find solutions to the Yamabe equation (\ref{originalYamabe}) by solving (\ref{Yamabe}). These solutions actually give solutions of the
Yamabe equation in more general situations, like the case when  $M$ is the base space of a harmonic Riemannian submersion treated in
\cite{BP1}, \cite{BP2}, \cite{Otoba}.

It is important to point out that positive  solutions of (\ref{Yamabe}) are the critical points of  the functional  $J_{\epsilon}:H^{1}(M)\rightarrow \re$, given by

\begin{equation}
J_{\epsilon}(u)= \epsilon^{-n} \int_{M} \left( \frac{1}{2} \epsilon^2 |\nabla u|^2 + \frac{1}{2} \left(\epsilon^2 s_g {c_{N}} +1\right)u^2-
\frac{1}{p_{N}}(u^+)^{p_{N}}\right)  d\mu_g,
\end{equation}

\noindent where $u^+(x)=\max \{u(x),0\}$.

We will build solutions
which have several peaks by using the Lyapunov-Schmidt reduction procedure which has been applied by several authors \cite{Dancer},
\cite{DKM}, \cite{Floer}, \cite{Li}, \cite{Micheletti}. In particular in \cite{Dancer} E. N.  Dancer, A. M. Micheletti, and A. Pistoia
apply  the procedure in a Riemannian n-manifold to build $k$-peaks solutions of the equation  

\begin{equation}
-\epsilon^{2} \Delta_gu+   u= u^{p-1},
\end{equation}

\noindent
for $p<p_n$, with the peaks approaching an isolated local minimum of the scalar curvature of the metric.
 We will apply similar techniques to study equation (\ref{Yamabe}).

\vspace{.5cm}

We will now briefly describe the construction. One first considers what will be called the limit equation in $\re^n$: recall that 
for $2<p<\frac{2n}{n-2}$,  $n>2$, the equation

\begin{equation}
\label{ED}
 -\Delta U + U=U^{p-1} \hbox{ in } \mathbb{R}^n
\end{equation}
has a unique (up to translations) positive solution  $U \in H^1(\mathbb{R}^n)$ that vanishes at infinity. Such function is radial and 
exponentially decreasing at infinity, namely 
\begin{equation}
\label{ED1}
 \lim_{|x|\to \infty} U(x)|x|^{\frac{n-1}{2}}e^{|x|}=c>0
\end{equation}
\begin{equation}
\label{ED2}
  \lim_{|x|\to \infty} | \nabla U (x) |  \ |x|^{\frac{n-1}{2}}e^{|x|}= c.
\end{equation}

\noindent
See reference \cite{Kwong} for details.  We will denote this solution by $U$ in the article, assuming that $p$ and 
$n$ are clear from the context.

Note that for any $\epsilon>0$,  the function $U_{\epsilon}(x)=U(\frac{x}{\epsilon})$, is a solution of 
\begin{equation}
-\epsilon^2 \Delta U_{\epsilon}+ U_{\epsilon}=U_{\epsilon}^{p-1}.
\end{equation}

For any $x \in M$ consider the exponential map $\exp_x : T_x M \to M$. Since $M$ is closed we can fix $r_0  > 0$ 
such that $\exp_x\big|_{B(0,r_0 )}: B(0, r_0  ) \to B_g (x, r_0 )$ is a diffeomorphism for any $x \in M$.
Here $B(0, r )$ is the ball in $\mathbb{R}^n$ centered at $0$ with radius $r$ and $B_g(x,r)$ is the 
geodesic ball in $M$ 
centered at $x$ with radius $r$.

Let $\chi_r$ be a smooth radial cut-off function such that $\chi_r(z)=1$ if $z \in B(0, r/2)$, $\chi_r (z) = 0$ if
$z \in \mathbb{R}^n \textbackslash B(0, r)$.

Fix any positive $r <r_0 $.  For a point $\xi \in M$ and $\ep > 0$ let us define the function  $W_{\ep,\xi}  :M \rightarrow 
\mathbb{R}$  by 
\begin{equation}
 W_{\ep,\xi}(x)= \left\{ 
	     \begin{array}{lcc}
             U_\ep(\exp_{\xi}^{-1}(x))\chi_r(\exp_{\xi}^{-1}(x)) &   if  &x\in B_g(\xi,r)  \\
             0 &   & otherwise \\
             \end{array}
   \right.
\end{equation}

One considers $ W_{\ep,\xi}$ as an approximate solution to equation (\ref{Yamabe}) which concentrates around $\xi$. 
As $\ep \rightarrow 0$ $ W_{\ep,\xi}$ will get more concentrated around $\xi$ and will
be closer to an exact solution. 
Summing up a finite number $k$  of these functions
concentrating on different points we have an approximate solution of equation (\ref{Yamabe}) which has $k$-peaks:
let $k_0\geq 1$ be a fixed integer and
denote $\overline{\xi}=(\xi_{1}, \dots,\xi_{k_0}) \in M^{k_0}$.

Then
\begin{equation}
\ V_{\epsilon,\overline{\xi}}:=\sum_{i=1}^{k_0}W_{\epsilon,\xi_{i}}.
\end{equation}

\noindent
is our approximate solution with $k_0$-peaks. We will find exact solutions by perturbing these approximate solutions.

Let $p=p_N$, ${\bf c}=c_{N}$, and 

\begin{equation}\label{beta}
\beta:=  {\bf c} \displaystyle\int_{\mathbb{R}^n} U^2(z)  \ dz  -  \frac{1}{n(n+2)}\displaystyle\int_{\mathbb{R}^n}  | \nabla U (z) |^2  | z |^2 \ dz.
\end{equation}

Note that the constant $\beta$ depends only on  $n$ and $m$.  We have not been able to find an analytical proof that $\beta \neq 0$ but at the end
of the article we give the numerical computation of $\beta$ for low values of $m$ and $n$. In all cases $\beta <0$.

Assuming that 
$\beta < 0$ we will show that for small $\epsilon$ and any isolated local minimum $x_0$ of $s_g$ there exists a  solution  of problem $(\ref{Yamabe})$ 
which is close to 
$V_{\epsilon,\overline{\xi}}$ in the norm $\|  \  \|_{\epsilon}$ defined by:

$$
\Vert u\Vert_{\epsilon}^{2}:=\frac{1}{\epsilon^{n}}\Big(\epsilon^{2}\int_{M}|\nabla_{g}u|^{2} \ d\mu_{g}+\int_{M}(\ep^2 {\bf c} s_g +1) \ u^{2} \ d\mu_{g}\Big)\ ,
$$

\noindent
with the points in $\overline{\xi}$ approaching  $x_0$:

\begin{thm}\label{main_thm} Assume that $\beta <0$. 
Let $\xi_0$ be an isolated local minimum  of the scalar curvature $s_g$. For each positive integer $k_0$, there exists $\ep_0= \ep_0(k_0) > 0$ such that for each
$\ep \in (0, \ep_0)$   there exist points $\xi_1^\ep, \dots, \xi_{k_0}^\ep\in M$,
\begin{equation}
\frac{d_g(\xi_i^\ep, \xi_j^\ep)}{\ep} \to +\infty \hspace{0.3cm} \hbox{ and } \hspace{0.3cm} d_g(\xi_0, \xi_j^\ep)\to 0,
\end{equation}

\noindent
and a solution $u_{\ep}$ of problem $(\ref{Yamabe})$ such that 
\begin{equation}
\Vert u_\ep - \sum_{i=1}^{k_0} W_{\ep,\xi_i^\ep}\Vert_{\epsilon} \to 0.
\end{equation}

\end{thm}

It has been proved by A. M. Micheletti and A. Pistoia \cite{Micheletti2} that for a generic Riemannian metric the critical points of the scalar curvature are non-degenerate
and in particular isolated. If $\beta >0$ one can prove the same theorem replacing the isolated local minimum of the scalar curvature
by an isolated local maximum.

As mentioned before, the theorem gives multiplicity results for the Yamabe equation on products (\ref{originalYamabe}) and certain
Riemannian submersions. Most of the known multiplicity results in these situations use bifurcation theory and assume that $s_g$ is constant
(\cite{BP1}, \cite{BP2}, \cite{Piccione}, \cite{Otoba}, \cite{Petean2}). The situation when $s_g$ is not constant, like in the theorem
was treated by J. Petean in \cite{Petean}, where it is proved that equation (\ref{Yamabe}) has $Cat(M) +1 $ ($Cat(M)$ is the Lusternik-Schnirelmann 
category of $M$)
solutions with low energy. To describe this result consider the Nehari manifold $N_{\epsilon}$, associated to $J_{\epsilon}$,

$$N_{\epsilon}=\{u \in H^{1}(M)-\{0\}: \int_{M} \left(  \epsilon^2 |\nabla u|^2 + \left( {\bf c}{s_g} \epsilon^2 +1\right)u^2\right)  d\mu_g= \int_{M}(u^+)^{p} d\mu_g \}.$$

Since for $u\in N_{\epsilon}$, $J_{\epsilon}(u)=\frac{1}{\epsilon^n} \frac{p-2}{2p}  \int_{M} (u^+)^{p}  d\mu_g$, it follows that  $J_{\epsilon}(u)$ restricted to $N_{\epsilon}$ is bounded below.  Set

$$m_{\epsilon}= \inf_{u \in N_{\epsilon}} J_{\epsilon}(u).$$

It is known   (cf. in \cite{Petean}, \cite{Akutagawa}), that, for these settings, 
	$$\lim_{\epsilon \rightarrow 0} m_{\epsilon} =\frac{p-2}{2p}|| U||_p^p.$$

 Given $\epsilon>0$ and $d>0$, we let,
$$ \Sigma_{\epsilon,d}=\{u \in N_{\epsilon}: J_{\epsilon}(u)<d\}.$$

The solutions in \cite{Petean} are in $ \Sigma_{\epsilon,m_{\epsilon} +\delta}$, for $\delta >0$ small.

Our next theorem shows
that those solutions have only one peak.

\begin{thm}
Let $\delta \in \re$ be such that $0<\delta<\frac{p-2}{2p}|| U||_p^p$. There exists $\epsilon_0>0$, such that if $\epsilon \in (0,\epsilon_0)$ and 
$u_{\epsilon}$ is a solution of  $(\ref{Yamabe})$,
with $u_{\epsilon} \in \Sigma_{\epsilon,  \frac{p-2}{2p}|| U||_p^p  + \delta}$, then   $u_{\epsilon}$ has a unique maximum point.
\end{thm}

In section 2 we will introduce notation and background and discuss a finite dimensional reduction of the problem.
In sections 3 and 4 we will prove
 Theorem 1.1 assuming the technical Proposition \ref{prop1}, which is proved in section 6.
In section 5 we will prove Theorem 1.2. We discuss the numerical calculation of the constant $\beta$ in the final appendix.
\\

\textbf{Acknowledgments.} The authors would like to thank Prof. Jimmy Petean for many helpful discussions on the subject. The second author was supported by program UNAM-DGAPA-PAPIIT IA106918.

\section{Approximate solutions and the reduction of the equation}

Positive solutions of (\ref{ED}) are the critical points of the functional $E:H^{1}(\re^n)\rightarrow \re$,
$$E(f)=\int_{\re^n} \left(\frac{1}{2}|\nabla f|^2 +\frac{1}{2} f^2-\frac{1}{p}(f^+)^p\right)dx.$$

Let $S_0 =\nabla E : H^{1}(\re^n)\rightarrow H^{1}(\re^n) $. $S_0 (U)=0$
and the solution $U$ is non-degenerate in the sense that Kernel$(S_0 '(U)) $ is spanned by 

$$
\psi^{i}(x):=\frac{\partial U}{\partial x_{i}}(x) 
$$
with $i=1,\dots, n.$

Note that for any $\epsilon>0$,  the function $U_{\epsilon}(x)=U(\frac{x}{\epsilon})$, is a solution of 
\begin{equation}
\label{Kwong2}
-\epsilon^2 \Delta U_{\epsilon}+ U_{\epsilon}=U_{\epsilon}^{p-1}.
\end{equation}

\noindent and so it is a critical point of  the functional 
$$E_{\epsilon}(f)=\epsilon^{-n}\int_{\re^n} \left(\frac{\epsilon^2}{2}|\nabla f|^2 +\frac{1}{2} f^2-\frac{1}{p}(f^+)^p\right)dx.$$

If ${S_{0}}_{\epsilon} = \nabla E_{\epsilon}$ then Kernel(${S_0}_{\epsilon}'(U_{\epsilon}) $) is spanned by the functions

$$
\psi_{\epsilon}^{i}(x):= \psi^{i}(\epsilon^{-1}x)
$$
with $i=1,\dots, n.$

Let us define on $M$ the functions 

\begin{equation}
 Z_{\ep,\xi}^{i}(x):= \left\{ \begin{array}{lcc}
             \psi_{\epsilon}^{i}(\exp_{\xi}^{-1}(x))\chi_{r}(\exp_{\xi}^{-1}(x)) &   \hbox{ if }  &x\in B_g(\xi,r)  \\
             \\ 0 &   & \hbox{ otherwise}. 
             \end{array}
   \right.
\end{equation}

Let $H_\ep$ be the Hilbert space $H^1_g(M)$ equipped with the inner product
$$
\langle u, v\rangle_{\ep}:=\frac{1}{\epsilon^{n}}\Big(\epsilon^{2}\int_{M}\nabla_{g}u\nabla_{g}v \ d\mu_{g}+\int_{M}(\ep^2 {\bf c}s_g +1) \ uv \ d\mu_{g}\Big)\ ,
$$
which induces the norm
$$
\Vert u\Vert_{\epsilon}^{2}:=\frac{1}{\epsilon^{n}}\Big(\epsilon^{2}\int_{M}|\nabla_{g}u|^{2} \ d\mu_{g}+\int_{M}(\ep^2 {\bf c} s_g +1) \ u^{2} \ d\mu_{g}\Big)\ .
$$

Similarly on $\re^n$ we define the inner product for  $u,v \in H^1_g(\re^n )$
$$
\langle u, v\rangle_{\ep}:=\frac{1}{\epsilon^{n}}\Big(\epsilon^{2}\int_{\re^n }\nabla u\nabla v \ dz +\int_{\re^n }  \ uv \ dz \Big)\ ,
$$
which induces the norm
$$
\Vert u\Vert_{\epsilon}^{2}:=\frac{1}{\epsilon^{n}}\Big(\epsilon^{2}\int_{\re^n }|\nabla u|^{2} \ dz +\int_{\re^n } \ u^{2} \ dz \Big)\ .
$$

It is important to note that $\Vert f_{\epsilon} \Vert_{\epsilon}$ is independent of $\epsilon$, where as before $f_{\epsilon} (x) = f(\frac{x}{\epsilon})$.

For $\epsilon >0$ and $\overline{\xi} =( \xi_1 ,..., \xi_{k_0} ) \in M^{k_0}$ let
$$K_{\epsilon,\overline{\xi}} :=\hbox{ span } \{Z_{\epsilon,\xi_{j}}^{i} : i=1, \dots, n, j=1, \dots, {k_0}\}$$
and

$$K_{\epsilon,\overline{\xi}}^{\perp}:=\{\phi\in H_{\epsilon}: \displaystyle<\phi, Z_{\epsilon,\xi_{j}}^{i}\displaystyle>_{\epsilon}=0, i=1, \dots, n, j=1 \dots, {k_0}\}.$$
Let $\Pi_{\epsilon,\overline{\xi}}$ : $H_{\epsilon}\rightarrow K_{\epsilon,\overline{\xi}}$ and $\Pi_{\epsilon\overline{\xi}}^{\perp}$: 
$H_{\epsilon}\rightarrow K_{\epsilon,\overline{\xi}}^{\perp}$ be the orthogonal projections. 
In order to solve equation (\ref{Yamabe}) we call 
$$S_{\epsilon} = \nabla J_{\epsilon} : H_{\epsilon} \rightarrow H_{\epsilon}.$$

\noindent
Equation (\ref{Yamabe}) is then $S_{\epsilon} (u) =0$.
The idea is that  the kernel of $S_{\epsilon}'( V_{\epsilon,\overline{\xi}} )$ should be close to $K_{\epsilon,\overline{\xi}}$ and then the 
linear map $\phi  \mapsto \Pi_{\epsilon\overline{\xi}}^{\perp} S_{\epsilon}'  ( V_{\epsilon,\overline{\xi}} )( \phi ) : K^{\perp} \rightarrow K^{\perp}$
should be invertible. Then  the inverse function theorem would imply that there is a unique small $\phi =\phi_{\epsilon ,\overline{\xi}} \in
K_{\epsilon,\overline{\xi}}^{\perp}$ such that
(this is the content of Proposition 2.1)

\begin{equation}
\label{Pi}
\Pi_{\epsilon,\overline{\xi}}^{\perp}\{ S_{\epsilon} ( V_{\epsilon,\overline{\xi}}+\phi ) \}=0 .
\end{equation}

And then we have to solve the finite dimensional problem

\begin{equation}
\label{PiP}
\Pi_{\epsilon,\overline{\xi}}\{      S_{\epsilon} ( V_{\epsilon,\overline{\xi}}+\phi )       \}=0.
\end{equation}

Consider  the function $\overline{J_{\epsilon}}$ : $M^{k_0}\rightarrow \mathbb{R}$ defined by
$$
\overline{J_{\epsilon}}(\overline{\xi}):=J_{\epsilon}(V_{\epsilon,\overline{\xi}}+\phi_{\epsilon,\overline{\xi}})\ .
$$

We will show in Proposition 3.1 that (\ref{PiP}) is equivalent to finding critical points of  $\overline{J_{\epsilon}}$.

\vspace{.5cm}

Let $\xi_{0}\in M$ be an isolated local minimum point of the scalar curvature. Let $k_0\geq 1$ be a fixed integer.
Given $\rho>0$, $\epsilon >0$ we consider the open set
\begin{equation}
 D_{\epsilon,\rho}^{{k_0}}:=\displaystyle\Big\{\overline{\xi} \in M^{k_0}  \hspace{0.2cm} / \hspace{0.2cm}
  \hspace{0.2cm} d_{g}(\xi_{0},\xi_{i})<\rho, i=1,\dots, {k_0}, \hspace{0.2cm}
\displaystyle\sum_{i\neq j}^{k_0}U_{\epsilon}\Big(\exp_{\xi_{i}}^{-1}\xi_{j}\Big)<\epsilon^{2}\Big\}. 
\end{equation}

Recall $U_{\epsilon}\Big(\exp_{\xi_{i}}^{-1}\xi_{j}\Big) = U  \Big( \epsilon^{-1} \exp_{\xi_{i}}^{-1}\xi_{j}\Big) $ and that
$U$ is a radial, positive, decreasing function. Then if 
$\overline{\xi}_{\epsilon} =({\xi_{\epsilon}}_1 , ...,{\xi_{\epsilon}}_{k_0} )\in D_{\epsilon,\rho}^{k_0}$ 
since $\| \exp_{\xi_{i}}^{-1}\xi_{j} \| = d_{g}({\xi_{\epsilon}}_{i},{\xi_{\epsilon}}_{j})$ we have that 
\begin{equation}\label{distance properties}
\lim_{\epsilon \rightarrow 0}  \frac{d_{g}({\xi_{\epsilon}}_{i},{\xi_{\epsilon}}_{j})}{\epsilon} = +  \infty .
\end{equation}

Moreover for any $\delta >0$ we have

\begin{equation}\label{distance properties2}
\lim_{\epsilon \rightarrow 0}  \frac{1}{\epsilon^2} \  e^{-(1+\delta )  \frac{d_{g}({\xi_{\epsilon}}_{i},{\xi_{\epsilon}}_{j})}{\epsilon} } = 0 .
\end{equation}

This follows from (\ref{ED1}): if we had $a>0$ and a sequence $\epsilon_i \rightarrow 0$ such that 

$$e^{-(1+\delta )  \frac{d_{g}({\xi_{\epsilon}}_{i},{\xi_{\epsilon}}_{j})}{\epsilon} } >a \epsilon^2 ,$$

then 

$$e^{-  \frac{d_{g}({\xi_{\epsilon}}_{i},{\xi_{\epsilon}}_{j})}{\epsilon} } >a \epsilon^2 e^{ \delta   \frac{d_{g}({\xi_{\epsilon}}_{i},{\xi_{\epsilon}}_{j})}{\epsilon} },$$

and applying (\ref{ED1}) to $ \epsilon^{-1} \exp_{\xi_{i}}^{-1}\xi_{j}$ since $U  \Big( \epsilon^{-1} \exp_{\xi_{i}}^{-1}\xi_{j}\Big) < \epsilon^2$ we get

$$\epsilon^2 \Big( \frac{d_{g}({\xi_{\epsilon}}_{i},{\xi_{\epsilon}}_{j})}{\epsilon} \Big)^{\frac{n-1}{2}} >c e^{-  \frac{d_{g}({\xi_{\epsilon}}_{i},{\xi_{\epsilon}}_{j})}{\epsilon} } > ca \epsilon^2 e^{ \delta   \frac{d_{g}({\xi_{\epsilon}}_{i},{\xi_{\epsilon}}_{j})}{\epsilon} },$$

\noindent
giving a contradiction.

We will prove:

\begin{prop}\label{prop1}
There exists $\rho_0>0, \epsilon_{0}>0, c>0$  and $\sigma>0$ such that for any $\rho\in(0,\ \rho_{0})$ , 
$\epsilon\in(0,\ \epsilon_{0})$  and $\overline{\xi}\in D_{\epsilon,\rho}^{k_0}$  there exists a unique
$\phi_{\epsilon,\overline{\xi}}=\phi(\epsilon,\ \overline{\xi}) \in
K_{\epsilon,\overline{\xi}}^{\perp}$ which solves equation $(\ref{Pi})$ and satisfies
\begin{equation}\label{desigualdad phi}
 \displaystyle \Vert\phi_{\epsilon,\overline{\xi}}\Vert_{\epsilon}\leq c\Big(\epsilon^{2}+\sum_{i\neq j} e^{-\frac{(1+\sigma)d_{g}(\xi_{i},\xi_{j})}{2\epsilon} }\Big).
\end{equation}
 Moreover, $\overline{\xi}\rightarrow\phi_{\epsilon,\overline{\xi}}$  is a $C^{1}$- map. Note that by (\ref{distance properties2}) $ \Vert\phi_{\epsilon,\overline{\xi}}\Vert_{\epsilon} =o(\epsilon )$.
\end{prop}

The proof of the proposition is  technical and follows the same lines used in
previous works, see \cite{Dancer, Floer, Micheletti}. For completeness we will sketch the proof following the proof
in \cite{Dancer},
but we pospone it to Section 6. In the next two sections we will prove Theorem 1.1 assuming this 
proposition.

\vspace{.5cm}

On the Banach space $\mathrm{L}_{g}^{q}(M)$ consider  the norm
$$
|u|_{q,\epsilon}:=\Big(\frac{1}{\epsilon^{n}}\int_{M}|u|^{q} \ d\mu_{g}\Big)^{1/q}
$$
%It is clear that for any $q\in[2,\ 2^{*}$) if $N\geq 3$ or $q\geq 2$ if $N=2$, the embedding $H_{\epsilon}rightarrow L_{\epsilon}^{q}$ is a continuous map.

Since $2<p<\frac{2n}{n-2}$ it follows from the usual Sobolev inequalities that there exists a constant $c$
independent of $\epsilon$ such that

\begin{equation}\label{norma_q}
 |u|_{p,\epsilon}\leq c\Vert u\Vert_{\ep}
\end{equation}
for any $u\in H_{\epsilon}$.

We denote by $L_{\epsilon}^{p}$ the Banach space $\mathrm{L}_{g}^{p}(M)$ with the norm $
|u|_{p,\epsilon}$. For $p' :=\displaystyle \frac{p}{p-1}$ the dual space ${L_{\epsilon}^{p}}^*$ is
identified with $L_{\epsilon}^{p'}$ with the pairing

$$ < \varphi , \psi > = \frac{1}{\epsilon^{n}}\int_{M}\varphi \psi $$

\noindent
for $\varphi \in L_{\epsilon}^{p}$, $\psi \in L_{\epsilon}^{p'}$.

 The embedding $\io_{\epsilon}$ : $H_{\epsilon}\rightarrow L_{\epsilon}^{p}$
is a compact continuous map and 
the adjoint operator $\io_{\epsilon}^{*}:L_{\epsilon}^{p'}\rightarrow H_{\epsilon} $, 
is a continuous map such that

$u=\io_{\epsilon}^{*}(v)\Leftrightarrow< \io_{\epsilon}^{*}(v),\displaystyle \varphi>_{\epsilon}=\frac{1}{\epsilon^{n}}\int_{M}v\varphi$, $\varphi\in H_{\epsilon} $
$\Leftrightarrow$ 

\begin{equation}
\label{Adjoint}
-\epsilon^{2}\triangle_{g}u+(\ep^2 {\bf c} s_g +1)u=v \hbox{ (weakly) on } M.
\end{equation}

 Moreover for the same constant $c$ in (\ref{norma_q}) we have that 
\begin{equation}\label{norma_ep}
 \Vert \io_{\epsilon}^{*}(v)\Vert_{\epsilon}\leq c|v|_{p',\epsilon}
\end{equation}
for any $v\in L_{\epsilon}^{p'}$.

Let 

$$f(u):=(u^{+})^{p-1}.$$

Note that 
\begin{equation}
S_{\epsilon} (u) = u- \io_{\epsilon}^{*}(f(u))\ ,\ u\in H_{\epsilon} ,
\end{equation}

\noindent
and we can rewrite problem $(\ref{Yamabe})$ in the equivalent way
\begin{equation}\label{equivalent problem}
 u=\io_{\epsilon}^{*}(f(u))\ ,\ u\in H_{\epsilon}.
\end{equation}
Note that a solution to (\ref{equivalent problem}) is a critical point of $J_{\epsilon}$ and
so it is a positive function. 

\vspace{1cm}

Now we will discuss some estimates related to the approximate solutions. 
The estimates are similar to ones obtained in \cite{Dancer, Micheletti} and we refer the reader to these articles for details.

The next lemma gives an explicit sense in which 
$W_{\epsilon , \xi}$ is an approximate solution of equation (\ref{Yamabe}):

\begin{lem} There exists a constant $c$ and $\epsilon_0 >0$ such that for any $\epsilon \in (0,\epsilon_0 )$, $\xi \in M$,
$$\| S_{\epsilon} (W_{\epsilon , \xi} ) \|_{\epsilon} \leq c \epsilon^2 $$
\end{lem}

\begin{proof} Let $Y_{\epsilon ,\xi} = -\epsilon^{2}\triangle_{g}W_{\epsilon ,\xi} +(\ep^2 {\bf c} s_g +1)W_{\epsilon ,\xi}$, so that by (\ref{Adjoint}) 
$W_{\epsilon ,\xi} = \io_{\epsilon}^{*}(Y_{\epsilon ,\xi} )$. Then 

$$\| S_{\epsilon} (W_{\epsilon , \xi} ) \|_{\epsilon}    =  \| \io_{\epsilon}^{*} (f  (W_{\epsilon , \xi} )) - W_{\epsilon , \xi} \|_{\epsilon}  $$

$$= \| \io_{\epsilon}^{*} (f  (W_{\epsilon , \xi} )  - Y_{\epsilon ,\xi}  )  \|_{\epsilon}  \leq c |f  (W_{\epsilon , \xi} )  - Y_{\epsilon ,\xi}  |_{p',\epsilon} $$

$$\leq c  |f  (W_{\epsilon , \xi} )  +  \epsilon^{2}\triangle_{g}W_{\epsilon ,\xi} - W_{\epsilon ,\xi}  |_{p',\epsilon} + c |  \ep^2 {\bf c} s_g  W_{\epsilon ,\xi}  |_{p',\epsilon} $$

But 

$$ | W_{\epsilon ,\xi}  |_{p',\epsilon} = \left( \epsilon^{-n}  \int_{B(0,r)} (U_{\epsilon} \chi_r )^{p'} \right)^{\frac{1}{p'}} \leq c 
\left(  \int_{B(0,r/\epsilon )} (U\chi_r  (\epsilon z ))^{p'}  dz \right)^{\frac{1}{p'}} $$

$$ \leq \bar{c}  \left( \int_{\mathbb{R}^{n}} U^{p'} dz \right)^{\frac{1}{p'}} \leq \bar{\bar{c}} .$$

Then 

$$|  \ep^2 s_g  W_{\epsilon ,\xi}  |_{p',\epsilon} \leq C \epsilon^2 .$$

In \cite[Lemma 3.3]{Micheletti} it is proved that 

$$  |f  (W_{\epsilon , \xi} )  + \epsilon^{2}\triangle_{g}W_{\epsilon ,\xi} - W_{\epsilon ,\xi}  |_{p',\epsilon} \leq C \epsilon^2 $$

\noindent
and the lemma follows.

\end{proof}

Since the function $U$ is radial it follows that if $i\neq j$ then $< \psi_{\epsilon}^i , \psi_{\epsilon}^j >_{\epsilon} =0$.

Then it is easy to see that for any $\xi \in M$ 

\begin{equation}\label{Zeta}
\lim_{\epsilon \rightarrow 0} \langle Z_{\epsilon,\xi}^{i},\ Z_{\epsilon,\xi}^{j}\rangle_{\epsilon}= \delta_{ij} \int_{\mathbb{R}^{n}}(|\nabla\psi^{l}|^{2}+(\psi^{l})^{2}) \ dz.
\end{equation}

Let us call $C=\displaystyle \int_{\mathbb{R}^{n}}(|\nabla\psi^{l}|^{2}+(\psi^{l})^{2}) \ dz$.

Given $\xi \in M$ and normal coordinates $(x_1 ,...,x_n )$ around $\xi$ it also follows that

\begin{equation}
\lim_{\epsilon \rightarrow 0} \epsilon \Big{\|}  \frac{\partial W_{\epsilon,\xi} }{\partial x^k}  \Big{\|}_{\epsilon}  = C,
\end{equation}

\begin{equation}\label{AA}
\lim_{\epsilon \rightarrow 0} \epsilon \Big{\langle} Z^i_{\epsilon , \xi}   , \frac{\partial W_{\epsilon,\xi} }{\partial x^k}  \Big{\rangle}_{\epsilon}  = \delta_{ik}C,
\end{equation}

and

\begin{equation}\label{DZ}
\lim_{\epsilon \rightarrow 0} \epsilon \Big{\|} \frac{\partial Z^i_{\epsilon , \xi} }{\partial x_k}  \Big{\|}_{\epsilon}  = \displaystyle \int_{\mathbb{R}^{n}}(|\nabla 
\frac{\partial \psi^{i}}{\partial x_k}|^{2}+(\frac{ \partial \psi^{i}}{\partial x_k})^{2}) \ dz
\end{equation}

The previous estimates deal with one peak approximations. For the multipeak approximate solutions
$V_{\epsilon , \overline{\xi} }$ with $ \overline{\xi} \in D_{\epsilon,\rho}^{{k_0}}$ consider normal coordinates 
$(x_1^i , ...,x_n^i )$ around each $\xi_i$ (i=1,...,$k_0$). 
Note that if $i\neq j$:

\begin{equation}\label{BB}
\vspace{0.5cm}
\displaystyle \frac{\partial}{\partial y_{h}^{j}}Z_{\epsilon,\xi_{i}(y^{i})}^{l}=\frac{\partial}{\partial y_{h}^{j}}W_{\epsilon,\xi_{i}(y^{i})}= 0 
\end{equation}

Also since the points are appropriately 
separated by (\ref{distance properties}) and the exponential decay of $U$ (\ref{ED1}), (\ref{ED2}), it follows that 
if $i\neq j$,

\begin{equation}\label{BBB}
\vspace{0.5cm}
\langle Z_{\epsilon,\xi_{j}}^{l}, \displaystyle \frac{\partial}{\partial y_{h}^{i}}W_{\epsilon,\xi_{i}(y^{i})}\rangle_{\epsilon}=o(1) 
\end{equation}

%\\\\\\\\\\\\\\\\\\\\\\\\\\\\\\\\\\\\\\\\\\\\\\\\\\\\\\\\\\\\\\\\\\\\\\\\\\\\\\\\\\\\\\\\\\\\\\\\\\\\\\\\\\\\\\\\\\\\\\\\\\\\\\\\\\\\\\\

%\\\\\\\\\\\\\\\\\\\\\\\\\\\\\\\\\\\\\\\\\\\\\\\\\\\\\\\\\\\\\\\\\\\\\\\\\\\\\\\\\\\\\\\\\\\\\\\\\\\\\\\\\\\\\\\\\\\\\\\\\\\\\\\\\\\\\\\\\\\\\\\\\
\section{The asymptotic expansion of $\overline{J}_\ep $}

For $\overline{\xi}\in D_{\epsilon,\rho}^{k_0}$  we consider the  unique
$\phi_{\epsilon,\overline{\xi}}=\phi(\epsilon,\ \overline{\xi}) \in
K_{\epsilon,\overline{\xi}}^{\perp}$ given by Proposition \ref{prop1} and define as in section 2, $\overline{J}_{\epsilon}(\displaystyle \overline{\xi})
= J_{\epsilon} (V_{\epsilon,\overline{\xi}}+\phi_{\epsilon,\overline{\xi}})$. 
In this section we will prove the following:

\begin{prop}\label{prop3} For $\overline{\xi} \in D_{\epsilon,\rho}^{k_0}$ we have
\begin{equation}
\label{expansion}
\overline{J}_{\epsilon}(\displaystyle \overline{\xi})=k_0\alpha + (1/2) \beta\epsilon^{2}\sum_{i=1}^{k_0}s_{g}(\xi_{i})-
\frac{1}{2}\sum_{i\neq j, i,j=1}^{k_0}\gamma_{ij}U(\displaystyle\frac{\exp_{\xi_{i}}^{-1}\xi_{j}}{\epsilon})+o(\epsilon^{2}), 
\end{equation}
$C^{0}$-uniformly with respect to $\overline{\xi}$ in compact sets of $D_{\epsilon,\rho}^{k_0}$ as $\epsilon$ goes to zero, 
where
\begin{equation}\label{alfa}
\al:= \frac{1}{2} \displaystyle\int_{\mathbb{R}^n} |\nabla U(z)|^2  \ dz +
\frac{1}{2}  \displaystyle\int_{\mathbb{R}^n} U^2(z)  \ dz - \frac{1}{p} \displaystyle\int_{\mathbb{R}^n} U^p(z) \ dz,
\end{equation}
\begin{equation}
\gamma_{ij}:= \int_{\re^n} U^{p-1}(z) e^{ \langle b_{ij}, z \rangle} dz
\end{equation}
with  $|b_{ij}|=1$
and
\begin{equation}
\beta:= {\bf c}\displaystyle\int_{\mathbb{R}^n} U^2(z)  \ dz  -  \frac{1}{n(n+2)}\displaystyle\int_{\mathbb{R}^n} |\nabla U(z)|^2  |z|^2 \ dz.
\end{equation}
 Moreover, if $\overline{\xi}_{\epsilon}$ is a critical point of $\overline{J_{\epsilon}}$, then the function
 $V_{\epsilon,\overline{\xi}_{\epsilon}}+\phi_{\epsilon,\overline{\xi}_{\epsilon}}$  is a solution to problem $(\ref{Yamabe}).$
\end{prop}

For a point $\xi \in M$ we will identify a geodesic ball around it with a ball in $\mathbb{R}^n$ by 
normal coordinates. We denote by $g_{ij}$ the expression of the metric $g$ in these coordinates and
consider the higher order terms in the Taylor expansions of the functions $g_{ij}$. See references \cite{G} and \cite{DKM} for details.
Let $\nabla$ and $R$ be the Riemannian connection and curvature operator of $M$.  Let
\[
 R_{ijkl}=\langle R (X_i,X_j)X_k, X_l\rangle  \hspace{0.2cm}\hbox{and}\hspace{0.2cm} R_{ij}=\langle R (X_i,X_j)X_i, X_j\rangle .
\]

%The metric coefficients at q := F (z) are given in terms of geometric data at ξ := F (0) and |z| := z 1 +. . .+z N 
%We now give the well known expansion for the metric in normal coordinates, we refer
%the reader to [11, 20] and some references therein for the expansion of the metric coefficients. 
%The expansions of the inverse of the metric and the volume element follows then from classical taylor expansions.

We will need the following lemma which is proved for instance in \cite{G}:

\begin{lem}
In a normal coordinates neighborhood of $\xi_{0} \in M$, the Taylor's series of $g$ around  $\xi_{0} $ is given by
\[
{g_\xi}_{ij}(z)= \delta_{ij} + \frac{1}{3}R_{kijl}(\xi)  z_k  z_l + O(|z|^3),
\] 
as $|z| \to 0$. Moreover,
\[
g^{ij}_\xi(z)=  \delta_{ij} - \frac{1}{3}R_{kijl}(\xi)  z_k  z_l + O(|z|^3).
\] 
Furthermore, the volume element on normal coordinates has the following expansion
\[
\sqrt{det \ g_\xi(z)}= 1 - \frac{1}{6} R_{kl}(\xi) z_k  z_l + O(|z|^3).
\] 
\end{lem}

%\[\big|g_{\xi_0}(\ep z)\big|^{1/2}= 1 - \frac{\epsilon^2}{4}\sum_{i,j,k} \frac{\partial^2 g_{\xi}^{ii}}{\partial z_j \partial z_k} (0) z_j z_k + O(\ep^3),\] 
%\[S_g(\exp_{\xi_0}(\epsilon z) )= S_g(\xi_0)+\ep \sum_{i,j,k} \frac{\partial^3 g_{\xi}^{ij}}{\partial z_j \partial z_i\partial z_k} (0) z_k -\ep \sum_{i,j,k} \frac{\partial^3 g_{\xi}^{ij}}{\partial z_k \partial z_j^2} (0) z_k + O(\ep^2),\] 
%because\[\frac{\partial}{\partial z_k}\Big|_{z=0} S_g(\exp_{\xi_0}(z) )= \sum_{i,j} \frac{\partial^3 g_{\xi}^{ij}}{\partial z_j \partial z_i\partial z_k} (0) z_k- \sum_{i,j} \frac{\partial^3 g_{\xi}^{ij}}{\partial z_k \partial z_j^2} (0) z_k\] 

\begin{lem}\label{JW}
For $\xi \in M$ and $\ep >0$ small we have
\begin{equation}
J_\ep (W_{\ep,\xi})= \al + \frac{\beta}{2} \ep ^2 s_g(\xi)  + o(\ep^2)
\end{equation}
\end{lem}

\begin{dem}
By direct computation
\[
J_\ep ( W_{\ep,\xi})=\frac{1}{\ep^n} \displaystyle\int_M \Big[\frac{1}{2}
\ep^2 \big|\nabla_g W_{\ep,\xi}\big|^2 +
\frac{1}{2}  (\ep^2 {\bf c}  s_g +1)  W_{\ep,\xi}^2 -
\frac{1}{p} \Big| W_{\ep,\xi}\Big|^p\Big]  \ d\mu_g
\]
\[
=\frac{1}{\ep^n} \displaystyle\int_M \Big[\frac{1}{2}
\ep^2 \big|\nabla_g W_{\ep,\xi}\big|^2 +
\frac{1}{2} W_{\ep,\xi}^2 -
\frac{1}{p} \Big| W_{\ep,\xi}\Big|^p\Big] \  d\mu_g
\]
\[
+
\frac{1}{\ep^n} \displaystyle\int_M \frac{1}{2}
\ep^2 {\bf c}s_g(x) W_{\ep,\xi}(x)^2 \ d\mu_g
= J + I
\]

We first estimate $I$. Let 
$
x=\exp_{\xi}(\epsilon z) 
$
with $z\in B(0,\frac{r}{\epsilon})$. Then doing the change of variables we obtain the expression

\[
2I=\ep^{-n} \ep^{2}  {\bf c} \displaystyle\int_{B_g(0,\frac{r}{\epsilon})}  s_g(\exp_{\xi}(\epsilon z) ) 
\Big( U(z)\chi_{r/\ep}(\epsilon z)\Big)^2 \sqrt{det \ g_{\xi}(\ep z)} \ \ep^{n}  \ dz
\]

By the exponential decay of $U$ (\ref{ED1}), we have

\[
2I=\ep^{2} {\bf c}\displaystyle\int_{B_g(0,\frac{r}{\sqrt{\epsilon}})}  s_g(\exp_{\xi}(\epsilon z) ) 
U(z)^2 \sqrt{det \ g_{\xi}(\ep z)} \   \ dz + o(\epsilon^2 )
\]

We consider  the Taylor's expansions of $\ g$ and $s_g$
around  $\xi$. For instance 

\[
s_g(exp_{\xi} (z)) = s_g  (\xi ) + \frac{\partial s_g}{\partial z_k} (\xi ) z_k  + O( |z|^2 )
\]

as $| z | \rightarrow  0$ . Therefore if $| z |  < \frac{r}{\sqrt{\ep}}$  for some fixed r > 0, then 

\[
s_g (exp_{\xi} (\ep z)) = s_g (\xi ) + \frac{\partial s_g}{\partial z_k} (\xi ) \ep z_k +  O(\ep ).
\]

Then

\[
2I=\ep^{2} {\bf c}s_g( \xi)  \displaystyle\int_{B_g(0,\frac{r}{\sqrt{\epsilon}})} 
U(z)^2   \ dz + o(\epsilon^2 )
\]

And using again  the exponential decay of $U$ we get

\begin{equation}
\label{N}
2I=\ep^{2} {\bf c}s_g(\xi) \Big( \displaystyle\int_{\mathbb{R}^n}  U^2(z) \ dz 
 \Big)
+ o(\ep^2)
\end{equation}

By Lemma $5.3$ of \cite{Micheletti} we have 
\begin{equation}\label{NN}
J= \al - \ep ^2 s_g(\xi)  \frac{1}{6}\displaystyle\int_{\mathbb{R}^n}\Big(\frac{u'(|z|)}{|z|}\Big)^2 z_1^4 \ dz + o(\ep^2)
\end{equation}

Here we are using that $U$ is a radial function, $U(z) = u(|z|) $  for a function $u:[0,+\infty ) \rightarrow \mathbb{R}$, and we identify 
$|\nabla U (z) | = |u' (|z| ) | $. Using polar coordinates to integrate

\[
\displaystyle\int_{\mathbb{R}^n}\Big(\frac{u'(|z|)}{|z|}\Big)^2 z_1^4 \ dz = \displaystyle\int_{0}^{+\infty}  \displaystyle\int_{\Sp^{n-1}(r)} \Big(\frac{u'(r)}{r}\Big)^2 z_1^4    \  dS(y) \ dr 
\]

\[
=  \displaystyle\int_{0}^{+\infty}  \Big(\frac{u'(r)}{r}\Big)^2 r^{n-1}\displaystyle\int_{\Sp^{n-1}}  (rz_1)^4   \  dS(y)  \ dr=
\displaystyle\int_{0}^{+\infty}  (u'(r))^2 r^{n+1}  \ dr  \ \  \displaystyle\int_{\Sp^{n-1}}  z_1^4  \  \ dS(y)
\]

For any homogeneous polynomial $p(x)$ of degree $d$ using the divergence theorem one obtains 
(see { Proposition $28$}  in \cite{Brendle} )

 \begin{equation}
  \displaystyle\int_{\Sp^{n-1}} p(x)  \ dS(x)=  \frac{1}{d(d+n-2)} \displaystyle\int_{\Sp^{n-1}} \Delta p(x) \  dS(x)
 \end{equation}

Then 

\[ 
 \displaystyle\int_{\Sp^{n-1}}  z_1^4  \  \ dS(z) = \frac{1}{4(n+2)}  \displaystyle\int_{\Sp^{n-1}}  12 z_1^2  \  \ dS(z)
=\frac{3}{n(n+2)}  \displaystyle\int_{\Sp^{n-1}}  \sum_{i=1}^n z_1^2  \  \ dS(z)
\]

\[
= \dfrac{3}{n(n+2)} V_{n-1} , 
\]

where $V_{n-1}$ is the volume of $\Sp^{n-1}$. Then we get 

\[
\displaystyle\int_{\mathbb{R}^n}\Big(\frac{u'(|z|)}{|z|}\Big)^2 z_1^4 \ dz = \frac{3}{n(n+2)} \displaystyle\int_{\mathbb{R}^n} |\nabla U|^2 |z|^2 dz
\]

\noindent
and using (\ref{N}), (\ref{NN}) the lemma follows. 

\end{dem}
\hfill$\square$
%%%%%%%%%%%%%%%%%%%%%%%%%%%%%%%%%%%%%
%%%%%%%%%%%%%%%%%%%%%%%%%%%%%%%%%%%%

\begin{lem}
\begin{equation}\label{Joverline}
 \overline{J_{\epsilon}}(\overline{\xi})=J_{\epsilon}(V_{\epsilon,\overline{\xi}}+\phi_{\epsilon,\overline{\xi}})=J_{\epsilon}(V_{\epsilon,\overline{\xi}})+o(\epsilon^{2})
\end{equation}
$C^{0}$- uniformly  in compact sets of $D_{\epsilon,\rho}^{k_0}$.
\end{lem}

\begin{dem}
If we let  $F(u)=\frac{1}{p}(u^+)^p$ then
$$
J_{\epsilon}(V_{\epsilon,\overline{\xi}}+\phi_{\epsilon,\overline{\xi}})-J_{\epsilon}(V_{\epsilon,\overline{\xi}})
$$
$$
=\frac{1}{2}\Vert\phi_{\epsilon,\overline{\xi}}\Vert_{\epsilon}^{2}+\frac{1}{\epsilon^{n}}\int_{M}[\epsilon^{2}\nabla_{g}
V_{\epsilon,\overline{\xi}}\nabla_{g}\phi_{\epsilon,\overline{\xi}}+(\ep^2 {\bf c}s_g + 1)V_{\epsilon,\overline{\xi}}\phi_{\epsilon,\overline{\xi}}-
f(V_{\epsilon,\overline{\xi}})\phi_{\epsilon,\overline{\xi}}].
$$

$$
-\frac{1}{\epsilon^{n}}\int_{M}[F(V_{\epsilon,\overline{\xi}}+\phi_{\epsilon,\overline{\xi}})
-F(V_{\epsilon,\overline{\xi}})-f(V_{\epsilon,\overline{\xi}})\phi_{\epsilon,\overline{\xi}}]
$$

Since $\phi_{\epsilon ,\overline{\xi}} \in
K_{\epsilon,\overline{\xi}}^{\perp}$ and it satisfies  (\ref{Pi}) 

$$0=\langle \phi_{\epsilon ,\overline{\xi}} , S_{\epsilon} (V_{\epsilon,\overline{\xi}} + \phi_{\epsilon ,\overline{\xi}} ) \rangle_{\epsilon} 
=\langle \phi_{\epsilon ,\overline{\xi}} ,  (V_{\epsilon,\overline{\xi}} + \phi_{\epsilon ,\overline{\xi}} ) - \io_{\epsilon}^{*} (f
 (V_{\epsilon,\overline{\xi}} + \phi_{\epsilon ,\overline{\xi}} ) ) \rangle_{\epsilon} $$

$$=\Vert\phi_{\epsilon,\overline{\xi}}\Vert_{\epsilon}^{2}+\frac{1}{\epsilon^{n}}\int_{M}[\epsilon^{2}\nabla_{g}
V_{\epsilon,\overline{\xi}}\nabla_{g}\phi_{\epsilon,\overline{\xi}}+(\ep^2 {\bf c}s_g + 1)V_{\epsilon,\overline{\xi}}\phi_{\epsilon,\overline{\xi}}-
f(V_{\epsilon,\overline{\xi}} +   \phi_{\epsilon,\overline{\xi}}  )\phi_{\epsilon,\overline{\xi}}.
$$

Therefore

$$J_{\epsilon}(V_{\epsilon,\overline{\xi}}+\phi_{\epsilon,\overline{\xi}})-J_{\epsilon}(V_{\epsilon,\overline{\xi}})
=-\frac{1}{2}\Vert\phi_{\epsilon,\overline{\xi}}\Vert_{\epsilon}^{2}+\frac{1}{\epsilon^{n}}\int_{M}
[f(V_{\epsilon,\overline{\xi}}+\phi_{\epsilon,\overline{\xi}})-f(V_{\epsilon,\overline{\xi}})]\phi_{\epsilon,\overline{\xi}}
$$
\begin{equation}\label{o_ep2}
- \displaystyle \frac{1}{\epsilon^{n}}\int_{M}[F(V_{\epsilon,\overline{\xi}}+\phi_{\epsilon,\overline{\xi}})
-F(V_{\epsilon,\overline{\xi}})-f(V_{\epsilon,\overline{\xi}})\phi_{\epsilon,\overline{\xi}}]
\end{equation}

By Proposition 2.1 $\Vert\phi_{\epsilon,\overline{\xi}}\Vert_{\epsilon}^2 = o(\epsilon^2 )$. By the mean value theorem 
we get for some $t_{1}, t_{2}\in[0$, 1 $]$
\begin{equation}\label{37}
\displaystyle \frac{1}{\epsilon^{n}}\int_{M}[f(V_{\epsilon,\overline{\xi}}+\phi_{\epsilon,\overline{\xi}})-
f(V_{\epsilon,\overline{\xi}})]\phi_{\epsilon,\overline{\xi}}=\frac{1}{\epsilon^{n}}\int_{M}f'(V_{\epsilon,\overline{\xi}}+t_{1}
\phi_{\epsilon,\overline{\xi}})\phi_{\epsilon,\xi}^{2}
\end{equation}
and
$$
\frac{1}{\epsilon^{n}}\int_{M}[F(V_{\epsilon,\overline{\xi}}+\phi_{\epsilon,\overline{\xi}})-F(V_{\epsilon,\overline{\xi}})-
f(V_{\epsilon,\overline{\xi}})\phi_{\epsilon,\overline{\xi}}]
$$
\begin{equation}\label{38}
 =\displaystyle \frac{1}{2\epsilon^{n}}\int_{M}f'(V_{\epsilon,\overline{\xi}}+t_{2}\phi_{\epsilon,\overline{\xi}})\phi_{\epsilon,\overline{\xi}}^{2}
\end{equation}

Moreover we have for any $t\in[0, 1]$
$$
\frac{1}{\epsilon^{n}}\int|f'(V_{\epsilon,\overline{\xi}}+t\phi_{\epsilon,\overline{\xi}})|\phi_{\epsilon,\overline{\xi}}^{2}\leq c\frac{1}{\epsilon^{n}}
\int V_{\epsilon,\overline{\xi}}^{p-2}\phi_{\epsilon,\overline{\xi}}^{2}+c\frac{1}{\epsilon^{n}}\int\phi_{\epsilon,\overline{\xi}}^{p}
$$
\begin{equation}\label{39}
 \displaystyle \leq c\frac{1}{\epsilon^{n}}\int\phi_{\epsilon,\overline{\xi}}^{2}+c\frac{1}{\epsilon^{n}}\int\phi_{\epsilon,\overline{\xi}}^{p}
\leq C(\Vert\phi_{\epsilon,\overline{\xi}}\Vert_{\epsilon}^{2}+\Vert\phi_{\epsilon,\overline{\xi}}\Vert_{\epsilon}^{p})=o(\epsilon^{2}).
\end{equation}
In the last inequality we use  (\ref{norma_q}) and the last equality follows from Proposition 2.1 . This proves the lemma. 
\end{dem}
\hfill$\square$

\begin{lem} For $\overline{\xi} \in D_{\epsilon,\rho}^{k_0}$ we have
\begin{equation} 
J_\ep (V_{\ep,\overline{\xi}})= {k_0}\al +\frac{1}{2} \beta \ep ^2 \sum_{i=1}^{k_0} s_g(\xi_i) - \frac{1}{2} 
\sum_{i\neq j}\gamma_{ij} U\Big(\frac{\exp_{\xi_j}^{-1}(\xi_i)}{\ep}\Big) + o(\ep^2)
\end{equation}
Here
\begin{center} 
$\displaystyle \gamma_{ij}\ :=\int_{\mathbb{R}^{n}}U^{p-1}(z)e^{\langle b_{ij},z\rangle} \ dz$,  
\end{center}
 where
\begin{center}
$b_{ij} :=\displaystyle \lim_{\epsilon\rightarrow 0}\displaystyle\frac{\exp_{\xi_{i}}^{-1}\xi_{j}}{|\exp_{\xi_{i}}^{-1}\xi_{j}|}$.  
\end{center}

\end{lem}

\begin{proof}
\noindent
\par

$J_\ep (V_{\ep,\overline{\xi}}) =
J_\ep \Big(\displaystyle\sum_{i=1}^{k_0} W_{\ep,\xi_i}\Big) =
$
\[
\frac{1}{\ep^n} \displaystyle\int_M \Big[\frac{1}{2}\ep^2 \big|\nabla_g\Big( \sum_{i=1}^{k_0} W_{\ep,\xi_i}\Big)\big|^2 +
\frac{1}{2}  (\ep^2 {\bf c}s_g +1) \Big(\sum_{i=1}^{k_0}  W_{\ep,\xi_i}\Big)^2 - \frac{1}{p} \Big(\sum_{i=1}^{k_0}  W_{\ep,\xi_i}\Big)^p\Big]  \ d\mu_g
\]
 \[
=\frac{1}{\ep^n}  \sum_{i=1}^{k_0} \Bigg[ \displaystyle\int_M \frac{1}{2}\ep^2\big|\nabla_gW_{\ep,\xi_i}\big|^2  \ d\mu_g +
\frac{1}{2} \displaystyle\int_M  (\ep^2 {\bf c}s_g +1)  \Big(W_{\ep,\xi_i}\Big)^2  \ d\mu_g - \frac{1}{p} \displaystyle\int_M  \Big(W_{\ep,\xi_i}\Big)^p  \ d\mu_g \Bigg]
\]
\[
 +\frac{1}{\ep^n} \sum_{i<j}^{k_0} \Bigg[ \displaystyle\int_M\ep^2 \nabla_g W_{\ep,\xi_i} \nabla_g W_{\ep,\xi_j}   \ d\mu_g+
\displaystyle\int_M  (\ep^2 {\bf c} s_g +1) W_{\ep,\xi_i} W_{\ep,\xi_j}   \ d\mu_g -  \displaystyle\int_M  \Big(W_{\ep,\xi_i}\Big)^{p-1}W_{\ep,\xi_j}  \ d\mu_g \Bigg]
\]
\[
 -\frac{1}{\ep^n} \Bigg[ \frac{1}{p} \displaystyle\int_M  \Big(\sum_{i=1}^{k_0}  W_{\ep,\xi_i}\Big)^p \ d\mu_g  -
 \frac{1}{p}  \sum_{i=1}^{k_0}  \displaystyle\int_M  \Big(W_{\ep,\xi_i}\Big)^p  \ d\mu_g - 
   \sum_{i< j}^{k_0} \displaystyle\int_M  \Big(W_{\ep,\xi_i}\Big)^{p-1}W_{\ep,\xi_j}  \ d\mu_g  \Bigg]
 \]
 \begin{equation}\label{i_1,2,3}
    =: I_1 + I_2 + I_3.
 \end{equation}
By Lemma \ref{JW} we get
 \[
 I_1= {k_0}\al + \frac{ \beta}{2} \ep ^2 \sum_{i=1}^{k_0} s_g(\xi_i) + o(\ep^2).
 \]
 Let us estimate the second term $I_2$ in (\ref{i_1,2,3}). We claim that  $I_2=o(\ep^2)$.

\[ I_2 =  \sum_{i<j}^{k_0}  \displaystyle\int_M  \ep^2 {\bf c}s_g  W_{\ep,\xi_i} W_{\ep,\xi_j}   \ d\mu_g  \ \  +
\]

\[\frac{1}{\ep^n} \sum_{i<j}^{k_0} \Bigg[ \displaystyle\int_M\ep^2 \nabla_g W_{\ep,\xi_i} \nabla_g W_{\ep,\xi_j}   \ d\mu_g+
\displaystyle\int_M   W_{\ep,\xi_i} W_{\ep,\xi_j}   \ d\mu_g -  \displaystyle\int_M  \Big(W_{\ep,\xi_i}\Big)^{p-1}W_{\ep,\xi_j}  \ d\mu_g \Bigg]
\]
 
It is easy to see that the first term is $o(\epsilon^2 )$. The second term only involves the $W_{\ep,\xi_j}$'s and it is explicitly estimated in 
\cite[Lemma 4.1]{Dancer}: it is shown there  that it is of the order of $o(\epsilon^2 )$. The term $I_3$ also only involves the
$W_{\ep,\xi_j}$'s  and it is estimated in  \cite[Lemma 4.1]{Dancer}. They show

\[
I_3 = - \frac{1}{2} 
\sum_{i\neq j}\gamma_{ij} U\Big(\frac{\exp_{\xi_j}^{-1}(\xi_i)}{\ep}\Big) + o(\ep^2)
\] 

This proves the lemma.

 \end{proof}

\begin{demProp3}
The last two lemmas prove (\ref{expansion}).
We are left to prove that if  $\overline{\xi}_{\epsilon} = (\xi_{1},\ \ldots\ ,\ \xi_{k_0})$ is a critical point of $\overline{J_{\epsilon}}$, then the function
$V_{\epsilon,\overline{\xi}_{\epsilon}}+\phi_{\epsilon,\overline{\xi}_{\epsilon}}$ is a
solution to problem  $(\ref{Yamabe})$. 
For $ \alpha=1, ... , k_0$ and $x^{\alpha}\in B(0, r)$ we  let $y^{\alpha} =  \exp_{\xi_{\alpha}}(x^{\alpha})$ and
$\overline{y}= (y^{1},\ \ldots,\ y^{k_0}) \in M^{k_0}$.

Since $\overline{\xi}$ is a critical point of $\overline{J_{\epsilon}}$,
\begin{equation}
 \label{A}
\displaystyle \frac{\partial}{\partial x_{i}^{\alpha}}\overline{J_{\epsilon}}(\overline{y}(x))\Big|_{x=0}=0, \hbox{ for } \alpha=1, ... , k_0,\ \  i=1, . . . , n.  
\end{equation}

We write

$$
S_{\epsilon} (V_{\epsilon,\overline{y}(x)}+\phi_{\epsilon,\overline{y}(x)}) = 
\Pi_{\epsilon , \overline{y}}^{\perp} S_{\epsilon} (V_{\epsilon,\overline{y}(x)}+\phi_{\epsilon,\overline{y}(x)}) +
\Pi_{\epsilon , \overline{y}} S_{\epsilon} (V_{\epsilon,\overline{y}(x)}+\phi_{\epsilon,\overline{y}(x)})
$$

The first term on the right is of course 0 by the construction of $\phi_{\epsilon,\overline{y}(x)}$. We write the second term as

$$\Pi_{\epsilon , \overline{y}} S_{\epsilon} (V_{\epsilon,\overline{y}(x)}+\phi_{\epsilon,\overline{y}(x)})
=\Sigma_{i,\alpha} C_{\epsilon}^{i,\alpha} Z_{\epsilon ,y^{\alpha} }^i $$

\noindent 
for some functions $C_{\epsilon}^{i,\alpha} : B(0,r)^{k_0} \rightarrow \re$. We have to prove that for each $i, \alpha$
(and $\epsilon >0 $ small), $C_{\epsilon}^{i,\alpha } (0)=0$.
Then fix $i, \alpha$.

We have 
\[ 0=
\displaystyle \frac{\partial}{\partial x_{i}^{\alpha}}\overline{J_{\epsilon}}(\overline{y}(x))=
J_{\epsilon}'(V_{\epsilon,\overline{y}(x)}+\phi_{\epsilon,\overline{y}(x)})[\frac{\partial}{\partial x_{i}^{\alpha}}
(V_{\epsilon,\overline{y}(x)}+\phi_{\epsilon,\overline{y}(x)})]=
\]
\[
=\langle S_{\epsilon} (V_{\epsilon,\overline{y}(x)}+\phi_{\epsilon,\overline{y}(x)}) ,
\frac{\partial}{\partial x_{i}^{\alpha}} \Big|_{x=0} (V_{\epsilon,\overline{y}(x)}+\phi_{\epsilon,\overline{y}(x)})\rangle_{\epsilon}
\]
\begin{equation}\label{B}
 =\langle \sum_{k,\beta}C_{\epsilon}^{k ,\beta} (0) Z_{\epsilon,y^{\beta}}^{k}, 
\displaystyle \frac{\partial}{\partial x_{i}^{\alpha}} \Big|_{x=0} (V_{\epsilon,\overline{y}(x)}+\phi_{\epsilon,\overline{y}(x)})\rangle_{\epsilon}.
\end{equation}

Since $\phi_{\epsilon,\overline{\xi}(y)}\in K_{\epsilon,\overline{\xi}(y)}^{\perp}$, for any $k$ and $\beta$ we have that
$\langle Z_{\epsilon,y^{\beta}}^{k}, \phi_{\epsilon,\overline{y}(x)}\rangle_{\epsilon}=0$. Then
$$ \liminf_{\epsilon \rightarrow 0}
|  \langle Z_{\epsilon,y^{\beta}}^{k},\ (\frac{\partial}{\partial x_{i}^{\alpha}}\phi_{\epsilon,\overline{y}(x)})\Big|_{y=0}\rangle_{\epsilon}   | 
 = \liminf_{\epsilon \rightarrow 0} |
-\langle(\frac{\partial}{\partial x_{i}^{\alpha}}Z_{\epsilon, y^{\beta}}^{k})\Big|_{y=0},\ \phi_{\epsilon,\overline{y} (x) }\rangle_{\epsilon} |
$$
\begin{equation}\label{D}
 \leq  \liminf_{\epsilon \rightarrow 0} \displaystyle \Vert(\frac{\partial}{\partial x_{i}^{\alpha}}Z_{\epsilon, y^{\beta} }^{k})\Big|_{y=0}
\Vert_{\epsilon}\cdot\Vert\phi_{\epsilon,\overline{y} (x) } \Vert_{\epsilon}=0,
\end{equation}
where the last equality follows from Proposition 2.1 and (\ref{DZ}).
\noindent

Now  from (\ref{BB})

\begin{equation}
 \langle \sum_{k,\beta}C_{\epsilon}^{k ,\beta} (0) Z_{\epsilon,y^{\beta}}^{k}, 
\displaystyle \frac{\partial}{\partial x_{i}^{\alpha}} \Big|_{x=0} V_{\epsilon,\overline{y}(x)}\rangle_{\epsilon}
\end{equation}

\begin{equation}
 =\langle \sum_{k,\beta}C_{\epsilon}^{k ,\beta} (0) Z_{\epsilon,y^{\beta}}^{k}, 
\displaystyle \frac{\partial}{\partial x_{i}^{\alpha}} \Big|_{x=0} W_{\epsilon,{y^{\alpha}}(x)}\rangle_{\epsilon}
\end{equation}

\begin{equation}
 =\langle \sum_{k}C_{\epsilon}^{k ,\alpha} (0) Z_{\epsilon,y^{\alpha}}^{k}, 
\displaystyle \frac{\partial}{\partial x_{i}^{\alpha}} \Big|_{x=0} W_{\epsilon,{y^{\alpha}}(x)}\rangle_{\epsilon}
+
\langle \sum_{k,\beta \neq \alpha}C_{\epsilon}^{k ,\beta} (0) Z_{\epsilon,y^{\beta}}^{k}, 
\displaystyle \frac{\partial}{\partial x_{i}^{\alpha}} \Big|_{x=0} W_{\epsilon,{y^{\alpha}}(x)}\rangle_{\epsilon}
\end{equation}

It follows from (\ref{BBB})  that 

\begin{equation}
\lim_{\epsilon \rightarrow 0}
\langle \sum_{k,\beta \neq \alpha}C_{\epsilon}^{k ,\beta} (0) Z_{\epsilon,y^{\beta}}^{k}, 
\displaystyle \frac{\partial}{\partial x_{i}^{\alpha}} \Big|_{x=0} W_{\epsilon,{y^{\alpha}}(x)}\rangle_{\epsilon} =0.
\end{equation}

Also

$$
 \langle \sum_{k}C_{\epsilon}^{k ,\alpha} (0) Z_{\epsilon,y^{\alpha}}^{k}, 
\displaystyle \frac{\partial}{\partial x_{i}^{\alpha}} \Big|_{x=0} W_{\epsilon,{y^{\alpha}}(x)}\rangle_{\epsilon}
= \langle C_{\epsilon}^{i  ,\alpha} (0) Z_{\epsilon,y^{\alpha}}^{i}, 
\displaystyle \frac{\partial}{\partial x_{i}^{\alpha}} \Big|_{x=0} W_{\epsilon,{y^{\alpha}}(x)}\rangle_{\epsilon}$$

$$+ \langle \sum_{k \neq i}C_{\epsilon}^{k  ,\alpha} (0) Z_{\epsilon,y^{\alpha}}^{k}, 
\displaystyle \frac{\partial}{\partial x_{i}^{\alpha}} \Big|_{x=0} W_{\epsilon,{y^{\alpha}}(x)}\rangle_{\epsilon}
$$

Then it follows from (\ref{AA}) that

$$
\lim_{\epsilon \rightarrow 0} \ \epsilon \
\langle \sum_{k}C_{\epsilon}^{k ,\alpha} (0) Z_{\epsilon,y^{\alpha}}^{k}, 
\displaystyle \frac{\partial}{\partial x_{i}^{\alpha}} \Big|_{x=0} W_{\epsilon,{y^{\alpha}}(x)}\rangle_{\epsilon}=C_{\epsilon}^{i  ,\alpha} (0)  C .
$$

And then it follows from  (\ref{B}) that $C_{\epsilon}^{i  ,\alpha} (0) =0$.

\end{demProp3}
\hfill$\square$

%%%%%%%%%%%%%%%%%%%%%%%%%%%%%%%%%%%%%%%%%%%%%%%
%%%%%%%%%%%%%%%%%%%%%%%%%%%%%%%%%%%%%%%%%%%%%%%%

%\\\\\\\\\\\\\\\\\\\\\\\\\\\\\\\\\\\\\\\\\\\\\\\\\\\\\\\\\\\\\\\\\\\\\\\\\\\\\\\\\\\\\\\\\\\\\\\\\\\\\\\\\\\\\\\\\\\\\\\\\\\\\\\\\\\\\\\\\\\\\\\\\

%\\\\\\\\\\\\\\\\\\\\\\\\\\\\\\\\\\\\\\\\\\\\\\\\\\\\\\\\\\\\\\\\\\\\\\\\\\\\\\\\\\\\\\\\\\\\\\\\\\\\\\\\\\\\\\\\\\\\\\\\\\\\\\\\\\\\\\\\\\\\\\\\\
%\\\\\\\\\\\\\\\\\\\\\\\\\\\\\\\\\\\\\\\\\\\\\\\\\\\\\\\\\\\\\\\\\\\\\\\\\\\\\\\\\\\\\\\\\\\\\\\\\\\\\\\\\\\\\\\\\\\\\\\\\\\\\\\\\\\\\\\\\\\\\\\\\
\section{Proof of Theorem $\ref{main_thm}$}

\begin{demTeo}
	
	We will prove that if $\bar \xi_{\epsilon} \in \overline{D_{\epsilon,\rho}^{k_0}}$, is such that 
$\overline{J}_{\epsilon}(\bar{\xi}_{\epsilon})= \max \{ \overline{J}_{\epsilon} (\bar{\xi}): \bar \xi \in \overline{D_{\epsilon,\rho}^{k_0}} \}$, 
 then  $\bar \xi_{\epsilon} \in {D_{\epsilon,\rho}^{k_0}}$. Then by Proposition 3.1 $u_{\epsilon}= V_{\epsilon ,\bar \xi_{\epsilon}}
+\phi_{\epsilon ,\bar \xi_{\epsilon}}$ is a solution to problem (3) and $\| u_{\epsilon} - V_{\epsilon ,\bar \xi_{\epsilon}} \|
=\| \bar \phi_{\epsilon ,\bar \xi_{\epsilon}}  \| = o(\epsilon )$.

We first construct a particular $\bar \eta_{\epsilon} \in D_{\epsilon,\rho}^{k_0} $. 	
Let $\bar \eta_{\epsilon} =(\eta_1,\eta_2,...,\eta_k)$, with $\eta_i = \eta_i(\epsilon)=\exp_{\xi_0}(\sqrt{\epsilon} \  e_i)$,
for $i \in \{1,2,...,k\}$, where $e_1,e_2,...,e_k$ are distinct points in  $\re^n$. 

Then, by direct computation, $\bar \eta_{\epsilon}$ verifies the following estimates:

a. $  d_g(\xi_0,\eta_i)=\sqrt{\epsilon} \ |e_i| $.

b.  $d_g(\eta_i,\eta_j)=|exp_{\eta_i}^{-1}\eta_j|=\sqrt{\epsilon} \  (|e_i-e_j|+o(1))$.

c. $U\left(\frac{\exp_{\eta_i}^{-1}\eta_j}{\epsilon}\right)=o(\epsilon^2)$, since
	
	$$U\left(\frac{\exp_{\eta_i}^{-1}\eta_j}{\epsilon}\right)=U\left(\frac{d_g(\eta_i,\eta_j)}{\epsilon}\right)=U\left(\frac{\sqrt{\epsilon} 
	(|e_i-e_j|+o(1))}{\epsilon}\right)=o(\epsilon^2).$$

We can then see that  $\bar J(\bar \eta_{\epsilon})=k_0 \alpha + (1/2)\beta  \epsilon^2 \sum_{i=1}^{k_0}s_g(\eta_i)+o(\epsilon^2)$, by   combining (c) 
and the expansion of $\bar J(\bar \eta_{\epsilon})$ in Proposition 3.1.

\noindent	Note that (a) and (c) imply that, for a fixed $\rho>0$ and $\epsilon$ small enough ,   $\bar \eta_{\epsilon} =(\eta_1,\eta_2,...,\eta_k) \in 
{D_{\epsilon,\rho}^{k_0}}$.

\noindent	Now, since $s_g(\xi_0)$ is a local minimum, for $\epsilon$ small we have an  expansion for $s_g(\eta_i)$: $$s_g(\eta_i)=s_g(\xi_0)+s_g''(\xi_0) \  
(d_g(\xi_0,\eta_i))^2+o(\sqrt {\epsilon}^3)=s_g(\xi_0)+s_g''(\xi_0) \ \epsilon \  |e_i|^2+o(\sqrt {\epsilon}^3), $$
	so in particular
	 $s_g(\eta_i)=s_g(\xi_0)+o(1).$
	 
Then we have:  
		$$\bar J(\bar \eta_{\epsilon})=k_0 \alpha +(1/2) \beta  \epsilon^2 \sum_{i}^{k_0}s_g(\eta_i)+o(\epsilon^2)  =k_0 \alpha +(1/2) \beta  
		\epsilon^2 \sum_{i}^{k_0}(s_g(\xi_0)+o(1))+o(\epsilon^2),$$
		\noindent  and we obtain
	\begin{equation}
	\label{etaexp}
\bar J(\bar \eta_{\epsilon})	=k_0 \alpha +k_0 (1/2)  \beta  \epsilon^2 s_g(\xi_0)+o(\epsilon^2). 
	\end{equation}

	Now, since $\bar{\xi}_{\epsilon}$ is a maximum of $\tilde{J_{\epsilon}}$ in $\overline{{D_{\epsilon,\rho}^{k_0}}}$, we have 
	\begin{equation}
	\label{max}
		\bar{J}_{\epsilon}(\bar{\xi}_{\epsilon})\geq \bar J(\bar \eta_{\epsilon}).
	\end{equation}

	Applying Proposition 3.1 to the left  side of (\ref{max}), we get
		$$k_0 \alpha + \frac{1}{2} \beta  \epsilon^2 \sum_{i}^{k_0}s_g(\xi_i)-\frac{1}{2}\sum_{i,j=1,i\neq j}^{k_0} \gamma_{ij} \  
		U\left(\frac{\exp_{\xi_i}^{-1}\xi_j}{\epsilon}\right) \geq k_0 \alpha +k_0 \frac{1}{2}  \beta  \epsilon^2s_g(\xi_0)+o(\epsilon^2),$$
	\noindent that is,
	
	 \begin{equation}
	 \label{upper}
	  \beta  \epsilon^2 \left(k_0 s_g(\xi_0)-\sum_{i}^{k_0}s_g(\xi_i) \right)+\sum_{i,j=1,i\neq j}^{k_0} \gamma_{ij} 
	  \ U\left(\frac{\exp_{\xi_i}^{-1}\xi_j}{\epsilon}\right)\leq o(\epsilon^2).
	 \end{equation}
	%We begin by using the expansion 

Fix $\rho$, small enough so that $\xi_0$ is the only  minimum of $s_g$ in $B_g(\xi_0,\rho)$.
With this choice of $\rho$ we  see that, in fact, each term in the left hand side of  inequality (\ref{upper}) is non-negative and therefore bounded 
from above by $o(\epsilon^2)$. 
	
	  Since $d(\xi_0,\xi_i)\leq \rho$, for each $i$, $1\leq i \leq k_0$,  we have,
$$	0\leq \beta  \epsilon^2  \left(k_0 s_g(\xi_0)-\sum_{i}^{k_0}s_g(\xi_i) \right)=o(\epsilon^2),$$
that is, since $\beta<0$,
\begin{equation}
\label{positiveS}
0\geq k_0 s_g(\xi_0)-\sum_{i}^{k_0}s_g(\xi_i) =o(1).
\end{equation}
	\noindent It follows that $\lim_{\epsilon\rightarrow 0}s_g(\xi_i)=s_g(\xi_0)$. And then, since $\xi_0$  
	is the only minimum point of $s_g$ in $B_g(\xi_0,\xi_i)$, we have $\lim_{\epsilon\rightarrow 0}\xi_i=\xi_0$. Hence, $\epsilon$ small enough implies 
	\begin{equation}
	\label {rho}
	d_g(\xi_i,\xi_0)<\rho.
		\end{equation}
	Now, recall  that $\gamma_{ij}:= \int_{\re^n} U^{p-1}(z) e^{ \langle b_{ij}, z \rangle} dz$, and that $|b_{ij}|=1$, for all $i,j \leq k$. 
	This implies that $\gamma_{ij}$ is bounded from below by a positive constant.  We define  
	$$\gamma:=\min \left\{\ \int_{\re^n} U^{p-1}(z) e^{ \langle b,z \rangle} dz : b\in \re^n, |b|=1 \right\}>0.$$

 Then, by (\ref{upper}) and (\ref{positiveS}), 
 
 $$o(\epsilon^2)\geq \sum_{i,j=1,i\neq j}^{k_0} \gamma_{ij} \ U\left(\frac{\exp_{\xi_i}^{-1}\xi_j}{\epsilon}\right)\geq  \sum_{i,j=1,i\neq j}^{k_0}\gamma
 \ U\left(\frac{\exp_{\xi_i}^{-1}\xi_j}{\epsilon}\right),$$
	
\noindent 	that is, for  $\epsilon$ small enough, 

\begin{equation}
\label{U}
U\left(\frac{\exp_{\xi_i}^{-1}\xi_j}{\epsilon}\right)<\epsilon^2.
\end{equation}	
Of course, (\ref{U}) and (\ref{rho})  imply that $\bar \xi_{\epsilon}\in D^{k_0}_{\epsilon,\rho}$.

\end{demTeo}

\section{Profile description of low energy solutions}

 Consider the Nehari manifold $N_{\epsilon}$, associated to $J_{\epsilon}$,

$$N_{\epsilon}=\{u \in H^{1}(M)-\{0\}: \int_{M} \left(  \epsilon^2 |\nabla u|^2 + \left({\bf c} {s_g}\epsilon^2 +1\right)u^2\right)  dV_g= \int_{M}(u^+)^{p} dV_g \}.$$

\noindent It is well known that the critical points of $J_{\epsilon}$ restricted to $N_{\epsilon}$ are positive solutions of (\ref{Yamabe}). 

Since for $u\in N_{\epsilon}$, $J_{\epsilon}(u)=\frac{1}{\epsilon^n} \frac{p-2}{2p}  \int_{M} (u^+)^{p}  dV_g$, it follows that  $J_{\epsilon}(u)$ restricted to $N_{\epsilon}$ is bounded below.  With this in mind we will set

$$m_{\epsilon}= \inf_{u \in N_{\epsilon}} J_{\epsilon}(u).$$

A similar setting on $\re^n$ is well known.

Consider the functional $E:H^{1}(\re^n)\rightarrow \re$,
$$E(f)=\int_{\re^n} \left(\frac{1}{2}|\nabla f|^2 +\frac{1}{2} f^2-\frac{1}{q}(f^+)^q\right)dx,$$

\noindent and the Nehari manifold $N(E)$, associated to $E$,

$$N(E)=\{f \in H^{1}(\re^n)-\{0\}: \int_{\re^n} \left(   |\nabla f|^2 + f^2\right)  dx= \int_{\re^n}(f^+)^{p} dx\}.$$

Of course, $U$ is the minimizer of the functional $E$, restricted to the Nehari manifold $N(E)$. We denote the minimum by $m(E)$, note that
\begin{equation}
\label{mE}
m(E)= \inf_{f \in N(E)} E(f)=\frac{p-2}{2p}|| U||_p^p.
\end{equation}

{\bf Remark}
	It is also a known result  (cf. in \cite{Petean}, \cite{Akutagawa}), that, for these settings, 
	$$\lim_{\epsilon \rightarrow 0} m_{\epsilon}=m(E).$$

%	\noindent We will call low energy solutions of eq. (\ref{Yamabe}), those solutions $u_{\epsilon}$ of eq. (\ref{Yamabe}) such that
%$J_{\epsilon}(u_{\epsilon})<2m(E)$.

Note that for any $\epsilon>0$,  the function $U_{\epsilon}(x)=U(\frac{x}{\epsilon})$, is a solution of 
\begin{equation}
\label{Kwong2}
-\epsilon^2 \Delta U_{\epsilon}+ U_{\epsilon}=U_{\epsilon}^{p-1}.
\end{equation}

\noindent With this in mind, we define the functional 
$$E_{\epsilon}(f)=\epsilon^{-n}\int_{\re^n} \left(\frac{\epsilon^2}{2}|\nabla f|^2 +\frac{1}{2} f^2-\frac{1}{p}(f^+)^p\right)dx,$$

\noindent and the Nehari manifold $N(E_{\epsilon})$, associated to $E_{\epsilon}$,

$$N(E_{\epsilon})=\{f \in H^{1}(\re^n)-\{0\}: \int_{\re^n} \left( \epsilon^2  |\nabla f|^2 + f^2\right)  dx= \int_{\re^n}(f^+)^{p} dx \}.$$

Of course, $U_{\epsilon}$ is a minimizer of the functional $E_{\epsilon}$, restricted to the Nehari manifold $N(E_{\epsilon})$. Note that 
\begin{equation}
\label{mE2}
m_{\epsilon}(E)= \inf_{f \in N(E_{\epsilon})} E_{\epsilon}(f)=m(E).
\end{equation}

In order to describe the profile of low energy solutions $u_{\epsilon}$ of equation (\ref{Yamabe}), for $\epsilon$ small, we start with the observation that 
they all have at least one maximum at some point $x_{\epsilon} \in M$.

 Given $\epsilon>0$ and $d>0$, we let,
$$ \Sigma_{\epsilon,d}=\{u \in N_{\epsilon}: J_{\epsilon}(u)<d\}.$$

\begin{lem}
	Let $u_{\epsilon}$ be a solution of (\ref{Yamabe}), such that $u_{\epsilon} \in \Sigma_{\epsilon, 2 m(E)}$. Then, for $\epsilon$ small enough,  $u_{\epsilon}$ is not constant.
\end{lem}

\begin{proof}
This follows directly from computation. If $u_{\epsilon}$ were constant, then 	we would have
	$$J_{\epsilon}(u_{\epsilon})= \epsilon^{-n} \int_{M} \left( \frac{1}{2} \left( {s_g}{\bf c}\epsilon^2 +1 \right) u_{\epsilon}^2- \frac{1}{p}(u_{\epsilon}^+)^{p}\right)  dV_g.$$
	
\noindent	That is, $J_{\epsilon}(u_{\epsilon})\rightarrow \infty$ as $\epsilon \rightarrow 0$, contradicting that  $J_{\epsilon}(u_{\epsilon}) < 2 m(E)$.
	
	\end{proof}

Using standard regularity theory it can proved that if $u_{\epsilon}$ is a non-negative  solution  of  (\ref{Yamabe}), such that $u_{\epsilon} \in \Sigma_{\epsilon,  2m(E)}$, then $u_{\epsilon} \in C^{2}(M)$ (see for example Theorem 4.1 of \cite{Parker}).

\begin{comment}
\begin{lem}
	\label{uisc2}
	Given $\epsilon>0$, let $u_{\epsilon} $ be a weak positive solution of  (\ref{Yamabe}) such that $J_{\epsilon}(u_{\epsilon})<2m(E)$. Then $u_{\epsilon} \in C^{2}(M)$. %, such that $J(u_{\epsilon})\leq m(E) + \delta$. Then, for $\epsilon$ small enough,  $u_{\epsilon}$ is not constant.
\end{lem}

\begin{proof}

	\end{proof}

\end{comment}
Since $M$ is compact, it follows that the low energy solutions $u_{\epsilon}$ have at least one maximum on $M$. Hence, if  $x_{\epsilon}$ is a maximum point of a solution $u_{\epsilon}$ of  (\ref{Yamabe}), then $\epsilon^2\Delta_g u(x_{\epsilon})\leq 0$, and then  $0 \geq \left({s_g(x_{\epsilon})}{\bf c}\epsilon^2 +1\right)u(x_{\epsilon})-u(x_{\epsilon})^{p-1}$. We thus have the following.

\begin{lem}
	\label{bigger}
	Let $x_{\epsilon} \in M$ be a maximum point of a solution $u_{\epsilon}$ of  (\ref{Yamabe}), with $u_{\epsilon} \in \Sigma_{\epsilon, 2 m(E)}$, then $u_{\epsilon}(x_{\epsilon})^{p-2} \geq {s_g(x_{\epsilon})}{\bf c}\epsilon^2 +1\geq {\min s_g}{\bf c}\epsilon^2 +1$.   %, such that $J(u_{\epsilon})\leq m(E) + \delta$. Then, for $\epsilon$ small enough,  $u_{\epsilon}$ is not constant.
\end{lem}

We now show that, locally, around a maximum point, low energy solutions $u_{\epsilon}$ are essentially the radial solution $U$ of equation (\ref{Yamabe}) on $\re^n$.
%%%%%%%%%%%%%%%%%%%%%%%%%%%%%%%%%%%%%%%%%%%%%%%%%%%%%%%%%%%%%%%%%%%%%%%%%%%%%%
%%%%%%%%%%%%%%%%%%%%%%%%%%%%%%%%%%%%%%%%%%%%%%%%%%%%%%%%%%%%%%%%%%%%%%%%%%%%%%%%

\begin{lem}
	\label{convergence}
	
	Let $\delta \in \re$ be such that $0<\delta<m(E)$.
	 For each $\epsilon>0$ denote by 	 $u_{\epsilon}$, $u_{\epsilon}:M\rightarrow \re$, a solution of  eq. (\ref{Yamabe}), such that $u_{\epsilon} \in \Sigma_{\epsilon,  m(E)+\delta}$. Denote  by $x_{\epsilon}$,  $x_{\epsilon}\in M$,  a maximum point of $u_{\epsilon}$. Then: %If $J_{\epsilon}(u_{\epsilon})\leq m(E)+\delta<2 \ m(E)$, 
	 %The following are satisfied.  %and $x_0 \in M$ be such that $x_{\epsilon_j}\rightarrow x_0 $.
	\begin{enumerate}
		\item Given any $\tilde R>0$ there is some $\tilde \epsilon>0$, such that if $\epsilon \in (0,\tilde \epsilon)$, then $u_{\epsilon}$ has only one maximum in    $B_g(x_{\epsilon},  \epsilon \tilde R)\subset M$.
		
		\item Given  $\eta \in (0,1)$, there is some $R_{\eta}>0$, and some $\epsilon_{\eta}>0$, such that if $\epsilon \in (0,\epsilon_{\eta})$,  then
\begin{equation}
	\label{ineq}
	J_{\epsilon}(u_{\epsilon}){\bigg|_{B_g(x_{\epsilon}, \epsilon R_{\eta})}}=\epsilon^{-n} \int_{{B_g(x_{\epsilon}, \epsilon  R_{\eta})}} \left( \frac{1}{2} \epsilon^2 |\nabla u_{\epsilon}|^2 + \frac{1}{2} \left({s_g}{\bf c}\epsilon^2 +1\right)u_{\epsilon}^2- \frac{1}{p}(u_{\epsilon}^+)^{p-1}\right)  dV_g
\end{equation}% ($J_{\epsilon}(u_{\epsilon_j})$ restricted to $B_g(x_{\epsilon_j}, r \epsilon_j)$).
	% \left(\frac{2p}{p-2}\right)
$$> \eta \  m(E).$$

	\end{enumerate}
\end{lem}
%%%%%%%%%%%%%%%%%%%%%%%%%%%%%%%%%%%%%%%%%%%%%%%%%%%%%%%%%%%%%%%%%%%%%%%
%%%%%%%%%%%%%%%%%%%%%%%%%%%%%%%%%%%%%%%%%%%%%%%%%%%%%%%%%%%%%%%%%%%%%%%%%%%%%%%
%The idea of the proof is to 

\begin{proof}
%Let $0<\delta<m(E)$. Recall from Remark \ref{limitm} that 	$J_{\epsilon}(u_{\epsilon})\rightarrow m(E)$ as $\epsilon\rightarrow 0$. Let $\epsilon_1$ be such that for $\epsilon \in (0,\epsilon_1)$ we have  $J_{\epsilon_j}(u_{\epsilon_j})\leq m(E)+\delta<2 \ m(E)$.

 Let 	 $\{u_{\epsilon_j}\}_{j \in \mathbb N}$, be any  sequnce of solutions of  eq. (\ref{Yamabe}),
 such that $\lim_{j \rightarrow \infty} \epsilon_j=0$ and such that $u_{\epsilon}\in N_{\epsilon}$ and  $J_{\epsilon}(u_{\epsilon})<m(E)+\delta$.
 Let $x_0 \in M$ be such that $x_{\epsilon_j}\rightarrow x_0 $.
	 
Consider a ball, $B_g(x_0,r)\subset M$, $r<\frac{r_0}{2}$, where $r_0$ is the injectivity radius of $M$.
Consider also the exponential function, $\exp_{x_0}$ on this ball. For $j$ big, $x_{\epsilon_j}\in B_g(x_0,r)$, 
and we define $y_j\in \re^n$ as $y_j=\exp_{x_0}^{-1}(x_{\epsilon_j})$, and the function $\bar v_j:B(0,\frac{r}{\epsilon_j}) \subset \re^n \rightarrow \re$, as 

\begin{equation}
\label{barv}
\bar v_j(z)=u_{\epsilon_j}\left(\   \exp_{x_0} (y_j+\epsilon_j z)\ \right)
\end{equation}	
	
	%\noindent for $z\in D_j=\{z \in \re^n | \ \ \  |y_j+ \epsilon_j  z|< r\}$. 
	\noindent Note that  $\bar v_j$ is well defined on $B(0,\frac{r}{\epsilon_j}) \subset \re^n$:
	recall that $y_j \rightarrow 0$ as $\epsilon_j \rightarrow 0$, so that for $j$ big enough,  $|y_j+ \epsilon_j  z|<r_0,$ since  $r<\frac{r_0}{2}$. 
	%, such that $|y_j+\epsilon_j z|<r_0$, where $r_0$ is the injectivity radius. 

	We now extend the domain of  $\bar v_j$ to all of $\re^n$. 
	
	Let $\chi_{r} (t):\re^n\rightarrow \re$ be a smooth cut-off function such that $\chi_{r} \leq1$, $\chi_{r} (t)=1$ for $t \in [0,r/2)$, 
$\chi_{r} (t)=0$ for $t \in [r,\infty)$, and $|\chi_{r}'(t)|\leq 2/t$, for $t \in [r/2,r)$. We will write $ \chi_j(z):=\chi_{r} (|\epsilon_j z|)$,
and then we define $v_j: \re^n \rightarrow \re$ as
	\begin{equation}
	\label{defvj}
	v_j(z)=  \bar v_j(z) \ \chi_j(z)=\bar v_j(z) \ \chi_{r} (|\epsilon_j z|),
	\end{equation}
	
\noindent for $z\in B(0,\frac{r}{\epsilon_j})\subset \re^n$, and as $v_j(z)=0$, for $z \notin B(0,\frac{r}{\epsilon_j})$.
	
	We  next prove that $v_j \rightarrow U$, $C^2_{loc}(\re^n)$, where $U$ is the positive, exponentially decreasing at infinity, solution of  equation $(\ref{Kwong2})$
	on $\re^n$, and the conclusions of the lemma will follow.

	 First note  that  $v_j$ is bounded in $H^{1}(\re^n)$, independently of $j\in \mathbb N$, for $j$ big:

	$$ ||v_j||_{H^{1,2}(\re^n)}^2=\int_{\re^n}\left(|\nabla v_j|^2+v_j^2\right)dz=\int_{B(0,\frac{r}{\epsilon_j})}\left(|\nabla v_j|^2+v_j^2\right)dz,$$
	\noindent  by construction, since $v_j= \bar v_j \ \chi_j $ (eq. $(\ref{defvj})$). Also, since $\chi_j^2\leq 1$ and 
	$|\nabla\chi_j|^2\leq \frac{4^2 \epsilon_j^2}{r^2}$, we have
	
	$$\int_{B(0,\frac{r}{\epsilon_j})}\left(\chi_j^2 |\nabla \bar v_j|^2+\bar v_j^2|\nabla \chi_j|^2+\chi_j^2 \bar v_j^2\right)dz\leq 
	\int_{B(0,\frac{r}{\epsilon_j})}\left( |\nabla \bar v_j|^2+ \bar v_j^2\right)dz + \frac{4^2 \epsilon_j^2}{r^2}\int_{B(0,\frac{r}{\epsilon_j})} \bar v_j^2dz$$
	$$\leq 2  \int_{B(0,\frac{r}{\epsilon_j})}\left( |\nabla \bar v_j|^2+ \bar v_j^2\right)dz, $$
	
	\noindent for $j$ big enough. Now, using a change of variables, $y=y_j+\epsilon_j z$, and recalling the definition of $\bar v_j$ (eq. $(\ref{barv})$) we have
	$$	\int_{B(0,\frac{r}{\epsilon_j})}\left( |\nabla \bar v_j|^2+ \bar v_j^2\right)dz=
	\frac{1}{\epsilon_j^n}\int_{B(0,r)}\bigg(\epsilon_j^2 |\nabla( u_{\epsilon_j}(  \exp_{x_0} (y)) )|^2+u_{\epsilon_j}^2 (\exp_{x_0} (y))\bigg)dy.  $$
	
	We now  use coordinates on $B_g(x_0,r)=\exp_{x_0}(B(0,r))$ to write,
	
	$$\frac{1}{\epsilon_j^n}\int_{B(0,r)}\left(\epsilon_j^2 |\nabla \bar v_j|^2+ \bar v_j^2\right)dy  \leq  c\frac{1}{\epsilon_j^n}\int_{M}\big(\epsilon_j^2 |\nabla_g  u_{\epsilon_j}|^2+\left({s_g}{\bf c}\epsilon_j^2 +1\right) u_{\epsilon_j}^2\big)dV_g, $$
	
	\noindent for some $c$ independent of $j$. Finally, by recalling that $u_{\epsilon_j}$ is in $N_{\epsilon_j}$, we have
	
	$$ \frac{1}{\epsilon_j^n}\int_M \left(\epsilon_j^2 |\nabla_g  u_{\epsilon_j}|^2+ \left({s_g}{\bf c}\epsilon_j^2 +1\right)u_{\epsilon_j}^2\right)dV_g = \frac{2p}{p-2} J_{\epsilon_j} (u_{\epsilon_j}),$$

	\noindent hence, since by hypothesis $J_{\epsilon}(u_{\epsilon})<2m(E)$, we get our bound:
		$$ ||v_j||_{H^{1,2}(\re^n)}^2\leq 2 c\frac{2p}{p-2} J_{\epsilon_j} (u_{\epsilon_j}) \leq  2c \frac{2p}{p-2}  \ 2 \  m(E). $$

%	\noindent Also, by Lemma \ref{bigger}, $u_{\epsilon_j}(x_{\epsilon_j})> (\frac{\min s_g}{a}\epsilon_j^2 +1)$, and then $v_j(0)> (\frac{\min s_g}{a}\epsilon_j^2 +1)$. 

%By the Rellich-Kondrakov Theorem the inclusion $H^{1} \subset L^p()$ is a compact operator.
	Then there exists some function $w\in H^{1}(\re^n)$,  such that $v_j\rightharpoonup w$ weakly in $H^{1}(\re^n)$. Also, given $\tilde R>0$,   $v_j\rightarrow w$ strongly in $L^p(B(0,\tilde R))$ (recall that $p=\frac{2(n+m)}{n+m-2}< \frac{2n}{n-2}$, so that the Sobolev Embedding (Theorem 7.26 in \cite{Trudinger}) is continuous and compact).
	
	Now, since  $u_{\epsilon_j}$ is a solution of equation (\ref{Yamabe}), we have:
	
\begin{equation}
\label{Yamabe2}
0=-\epsilon_j^2\Delta_g u_{\epsilon_j} + \left({s_g}{\bf c}\epsilon^2 +1\right)u_{\epsilon_j}-u_{\epsilon_j}^{p-1}. 
\end{equation}
	Given $\varphi \in C_0^{\infty}(\re^n)$, take $\epsilon_j$ small enough so that $supp \ \varphi \subset B(0, \frac{1}{2}\frac{r}{\epsilon_j})$. Also, choose normal coordinates in $B_g(x_0,r)=\exp(B(0, r))$, such that $g_{il}(0)=\delta_{il}$.
	
	We use these coordinates to rewrite eq. (\ref{Yamabe2}) on $B(0, \frac{1}{2} \frac{r}{\epsilon_j})\subset \re^n$ (recall that, for $z\in B(0, \frac{r}{\epsilon_j})$,   $v_j(z)= u_{\epsilon_j}(exp_{x_0}(y_j+ \epsilon_j z))$): 
		\begin{equation}
	\label{YamabeCoord}
	0=  \frac{1}{|g(y_j + \epsilon_j z)  |^{1/2}}  \sum_{il} \partial_{z_l}\Big(|g(y_j + \epsilon_j z)  |^{1/2}  g^{il}(y_j+\epsilon_j z) \ \partial_{z_i} \ v_j(z)\Big)  + \Big({s_g}{\bf c}\epsilon^2 +1\Big)v_{j}(z)-v_{j}(z)^{p-1}
\end{equation}

	Multiplying (\ref{YamabeCoord}) by $\varphi$ and integrating on $\mathit{supp} \ \varphi \subset B(0, \frac{1}{2}\frac{r}{\epsilon_j})\subset \re^n$, yields,
	
	$$ 0=\int_{supp \ \varphi}  \left(  g^{il}(y_j + \epsilon_j z)  \ \partial_{z_i} v_j(z) \  \partial_{z_l} \varphi(z) + 
	({s_g}{\bf c}\epsilon_j^2 +1) v_j(z) \varphi(z) -  v_j(z)^{p-1} \varphi(z)\right) \Big|g(y_j + \epsilon_j z)  \Big|^{1/2} \ dz$$
	\noindent Taking into account that, for each $z\in \re^n$, 
	$$\lim_{j\rightarrow \infty} g_{il}(y_j + \epsilon_j z)  = g_{il}(0)=\delta_{il}.$$
	\noindent and taking into account the weak convergence $v_j \rightarrow w$, on $ H^{1,2}(\re^n)$,  and strong in $L^q(B(0,\tilde R))$, for $2\leq q< \frac{2n}{n-2}$, we have
			$$  0= \int_{supp \ \varphi}  \left(  \delta_{il}\partial_{z_i} w(z) \  \partial_{z_l} \varphi(z)  +w(z)\varphi(z)- w^{p-1}(z) \varphi(z) \right) \ dz,$$
			
			\noindent that is
		$$  0= \int_{supp \ \varphi}  \left(	\left<\nabla w(z),\nabla \varphi(z)\right> +w(z)\varphi(z)- w^{p-1}(z) \varphi(z) \right) \ dz.$$
	
	Thus, $w$ weakly solves 
$	-\Delta w + w=w^{p-1}$ on $\re^n$.
	
	We now use regularity theory to prove that, moreover, $v_j \rightarrow w$ in $C_{loc}^2(\re^n)$.
	
We use normal coordinates again. By equation (\ref{YamabeCoord}), for $z \in B(0,\frac{r}{\epsilon_j}$):

	$$ \frac{1}{|g(y_j + \epsilon_j z)  |^{1/2}}  \sum_{il} \partial_{z_l}\left(|g(y_j + \epsilon_j z)  |^{1/2}  g^{il}(y_j+\epsilon_j z) \ \partial_{z_i} \ v_j(z)\right)$$
	
	$$= - ({s_g}{\bf c}\epsilon^2 +1)v_{j}(z) +v_j ^{p-1}(z)$$
Given $\tilde R>0$, for $j$ big we have $\frac{r}{2 \epsilon_j}>\tilde R$.	
 We define on $B(0,\tilde 2 R) \subset \re^n$, for each $j$:
	$$f_j(z):=|g(y_j + \epsilon_j z)  |^{1/2} \left(- ({s_g}{\bf c}\epsilon^2 +1) \ v_{j}(z) + \ v_j ^{p-1}(z) \right),$$
	
	\noindent and the strictly elliptic operator, $\square$ :
	$$\square \  v_j(z)= \sum_{il} \partial_{z_l}\left(|g(y_j + \epsilon_j z)  |^{1/2}  g^{il}(y_j+\epsilon_j z) \ \partial_{z_i} \ v_j(z)\right).$$
	\noindent Note that for each $j$, $\square  v_j=f_j$.% and that $\square $ is a strictly elliptic operator in  $ B(0,2 \tilde R)$.
	
	  By local elliptic regularity (see for example, Theorem 8.10 in \cite{Trudinger})
	$$|| v_j||_{H^{2,\frac{p}{p-1}}(B(0,\tilde R))}\leq C_1 \left(||v_j||_{L^{\frac{p}{p-1}}(B(0,\tilde R))} + ||f_j||_{L^{\frac{p}{p-1}}(B(0,\tilde R))} \right) $$
		\noindent for $j$ big  such that $\frac{r}{2 \epsilon_j}>\tilde R$.
On the other hand, since we have strong convergence of $v_j$ in $L^p(B(0,\tilde R))$, the $f_j$ are uniformly bounded in $L^{\frac{p}{p-1}}(B(0,\tilde R))$
(recall that $p=\frac{2(n+m)}{n+m-2}< \frac{2n}{n-2}$). And then, 
we have a uniform bound,
$$|| v_j||_{H^{2,\frac{p}{p-1}}(B(0,\tilde R))}\leq C_1 (||v_j||_{L^{\frac{p}{p-1}}(B(0,\tilde R))} + ||f_j||_{L^{\frac{p}{p-1}}(B(0,\tilde R))} ) \leq C_2.$$
	
 In turn, by the Sobolev Embedding Theorem (see for example (Theorem 7.26 in \cite{Trudinger})
this would imply that $v_j$ are uniformly bounded in $ L^{q'} (B(0,\tilde R))$ with $q'= \frac{q n}{n-2q}$ and $q=\frac{p}{p-1}$. 
And then, applying local elliptic regularity again to $\square  v_j=f_j$, the $ v_j$ would be uniformly bounded in ${H^{2,\frac{q'}{q'-1}}(B(0,\tilde R))}$.  
Since $q'>q$ in this process, we can continue this bootstrap argument to prove that for any $s$ large, $f_j$ is uniformly bounded in $  {L^{s}(B(0,\tilde R))}$ 
and then, also $ v_j$ in $H^{2,s}(B(0,\tilde R))$. 
	
	It follows, by the second part of the Sobolev Embedding Theorem (Theorem 7.26 in \cite{Trudinger}),  $v_j$ is uniformly bounded
	in $C^{0,\theta}(B(0,\tilde R))$, for $0<\theta<1$, and in consequence, $f_j$ is uniformly bounded as well.

	Then we make use of local elliptic regularity again for $\square  v_j=f_j$: by the Schauder estimates (see for example, 
	Thm 6.2 in \cite{Trudinger}) $v_j\in C^{2,\theta}(M)$ and:
	
	$$|| v_j||_{C^{2,\theta}(B(0,\tilde R))}\leq C_3 ( ||v_j||_{C^{0,\theta}(B(0,\tilde R))} + ||f_j||_{C^{0,\theta}(B(0,\tilde R))})  \leq C_4,$$
	
	\noindent this implies that $ v_j \rightarrow w$ in $C^2(B(0,\tilde R))$. And by Lemma \ref{bigger}, $w(0)\geq{\min s_g }{\bf c} \epsilon_j +1$, 
	i.e. $w\neq 0$. Of course, this implies that $w=U$, the radial, positive solution of $(\ref{Kwong2})$ on $\re^n$.

	Recall that, for $z\in B(0,\tilde R)\subset  B(0,\frac{r}{2 \epsilon_j})$,   
	$$v_j(z)=\chi_j \bar v_j(z)=\bar v_j(z)= u_{\epsilon_j}(exp_{x_0}(y_j+ \epsilon_j z)).$$ {n
We conclude that  $u_{\epsilon_j}$ has only one maximum in $B_g(x_{\epsilon_j}, \epsilon_j \tilde R)$ since	 $U$ has only one maximum and  convergence 
of $v_j$ to $U$ is $C^2$ in $B(0,\tilde R)$. Since $\tilde R$ was arbitrary, the first conclusion of the Lemma follows.

	We now prove the second part of the Lemma.  Let $\eta\in (0,1)$. Then there is some $R_{\eta}>0$  such that 
	
\begin{equation}
\int_{B(0,R_{\eta})} (|\nabla U|^2+ U^2) dz>  \eta\frac{2p}{p-2}  (m(E)).
\end{equation}

Since we have local $C^2$ convergence of $v_j$ to $U$, we have, as $j\rightarrow \infty$,
	
\begin{equation}
\label{conv}
\int_{B(0,R_{\eta})} \left(|\nabla v_{j}|^2+ \left({s_g}{\bf c}\epsilon_j^2 +1\right)v_j^2 \right)dz \rightarrow \int_{B(0,R_{\eta})} (|\nabla U|^2+ U^2) dz.
\end{equation} 
%> \eta \ (m(E)).$$

Let $s_j={s_g}{\bf c}\epsilon_j^2 +1$ (note that $s_j\rightarrow 1$ as $j\rightarrow \infty$). 

	\noindent  By construction, $v_j= \bar v_j \ \chi_j $ (eq. (\ref{defvj})). Moreover,  recall that  $\chi_j(z)=\chi_r(\epsilon_j z)=1$ 
	for $\epsilon_j z\leq \frac{r}{2}$ i.e. for $ z\leq \frac{r}{2 \epsilon_j}$. That is for $j$ big enough,			
	$\chi_j(z)=1$ in $B(0,R_{\eta})$ and then
				$$ \int_{B(0,R_{\eta})} \left(|\nabla v_{j}|^2+ s_j v_j^2 \right)dz =\int_{B(0,R_{\eta})}\left( |\nabla \bar v_j|^2+s_j  
				\bar v_j^2\right)dz.$$
			
		 Using a change of variables, $y=y_j+\epsilon_j z$, and recalling the definition of $\bar v_j$ (eq. (\ref{barv})) we have
	\begin{equation}
	\label{vj}
	\int_{B(0,R_{\eta})}\left( |\nabla  v_j|^2+ s_j  v_j^2\right)dz=
	\frac{1}{\epsilon_j^n}\int_{B(y_j,R_{\eta}\epsilon_j)}\bigg(\epsilon_j^2 |\nabla( u_{\epsilon_j}(  \exp_{x_0} (y)) )|^2+s_j u_{\epsilon_j}^2 (\exp_{x_0} (y))\bigg)dy.
	\end{equation}

	We make use again of normal coordinates in $B_g(x_0, r)=\exp_{x_0}(B(0, r))$, such that $g_{il}(0)=\delta_{il}$. By continuity, of $g_{il}$, for each $\alpha>0$, there is some $j_0$ such that 
	\begin{equation}
	\label{normal}
	(1+\alpha)^{-1} \delta_{il} dx^i dx^l \leq g^{il}dx^idx^l\leq (1+\alpha)\delta_{il}dx^idx^l
	\end{equation}
	$$(1+\alpha)^{-1} dy \leq dV_g\leq (1+\alpha)dy$$
	
	\noindent in $B(y_{j},\epsilon_j R_{\eta})$, for $j>j_0$. That is, we may construct a sequence $\{\alpha_j\}_{j\in \mathbb n}$, such that $\alpha_j \rightarrow 0$ and 
		\begin{equation}
	\label{normalj}
	(1+\alpha_j)^{-1} \delta_{il} dx^i dx^l \leq g^{il}dx^idx^l\leq (1+\alpha_j)\delta_{il}dx^idx^l
	\end{equation}
	$$(1+\alpha_j)^{-1} dy \leq dV_g\leq (1+\alpha_j)dy$$
	
	\noindent in $B(y_{j},\epsilon_j R_{\eta})$. And then

		$$	\frac{1}{\epsilon_j^n}	\int_{B_g(x_j,R_{\eta}\epsilon_j)}\bigg(\epsilon_j^2 |\nabla_g( u_{\epsilon_j}(  \exp_{x_0} (y)) )|^2+s_j u_{\epsilon_j}^2 (\exp_{x_0} (y))\bigg)dV_g$$
		$$\geq
(1+\alpha_j)^{-2}\frac{1}{\epsilon_j^n}\int_{B(y_j,R_{\eta}\epsilon_j)}\bigg(\epsilon_j^2 |\nabla( u_{\epsilon_j}(  \exp_{x_0} (y)) )|^2+s_j u_{\epsilon_j}^2 (\exp_{x_0} (y))\bigg)dy.  $$

This last inequality, and (\ref{vj}) yield,% for $\epsilon_j$ small enough,

$$ J_{\epsilon_j}(u_{\epsilon_j})\Big|_{B_g(x_j, \epsilon_j R_{\eta})} = \left(\frac{1}{2}-\frac{1}{p}\right)\frac{1}{\epsilon_j^n} \int_{B_g(x_j,\epsilon_j 
R_{\eta})} \left(\epsilon_j^2|\nabla u_{\epsilon_j}|^2+ \left({s_g}{\bf c}\epsilon_j^2 +1\right)u_{\epsilon_j}^2 \right)dV_g$$

$$\geq \frac{p-2}{2p} (1+\alpha_j)^{-2} \frac{1}{\epsilon_j^n} \int_{B(y_j,R_{\eta}\epsilon_j)}\bigg(\epsilon_j^2 |\nabla( u_{\epsilon_j}
(  \exp_{x_0} (y)) )|^2+s_j u_{\epsilon_j}^2 (\exp_{x_0} (y))\bigg) dy  $$

	$$ = \frac{p-2}{2p} (1+\alpha_j)^{-2}  \int_{B(0,R_{\eta})} \left(|\nabla v_{j}|^2+ \left({s_g}{\bf c}\epsilon_j^2 +1\right)v_j^2 \right)dz.$$
	
	Recall that, by (\ref{conv}) and (\ref{normalj}), 
	
$$	\lim_{ j \rightarrow \infty } \frac{p-2}{2p} (1+\alpha_j)^{-2}  \int_{B(0,R_{\eta})} \left(|\nabla v_{j}|^2+ \left({s_g}{\bf c}\epsilon_j^2 +1\right)v_j^2 \right)dz$$
	$$ = \frac{p-2}{2p} \int_{B(0,R_{\eta})} (|\nabla U|^2+ U^2) dz> \eta \ m(E). $$

	We conclude that for $j$ big enough,
	
	$$ J_{\epsilon_j}(u_{\epsilon_j}) \Big|_{B_g(x_j, \epsilon_j R_{\eta})} > \eta \ m(E).$$
	
}		\end{proof}

Using lemma \ref{convergence}, we prove Theorem 1.2

\begin{proof}

If the theorem were not true, there would exist a sequence $\epsilon_j\rightarrow 0$ and corresponding solutions $u_{\epsilon_j}$ with at least two local maxima. We will denote them
by $x^1_{\epsilon_j}$ and $x^2_{\epsilon_j}$. 
%Of course, for $j$ big enough, $\epsilon_j<\epsilon_1$ for each $j$, which implies,

%We will denote also $x^1$ and $x^2$, points in $M$ such that $x^i_{\epsilon_j}\rightarrow x^i$, as $j\rightarrow \infty$, $i=1,2$. Note that it may be that $x^1=x^2$.

%\begin{itemize} %[label={}]
%	\item Case 1.
Since $m(E)+\delta <2 {m(E)}$, then $\frac{1}{2}+\frac{\delta}{2 m(E)}<1$. Let  $\eta>0$, be such that  $\frac{1}{2}+\frac{\delta}{2 m(E)}<\eta<1$. For this $\eta$ there is some $R_{\eta}$ and $\epsilon_{\eta}$, as in part (2) of Lemma \ref{convergence}.
Let $\tilde{R}=2 R_{\eta}$ as in part (1) of Lemma \ref{convergence}, with its corresponding $\tilde \epsilon$. Let $\epsilon_0=\min\{\epsilon_{\eta},\tilde \epsilon\}$.

	Then, we consider the open balls  $B_g(x_{\epsilon_j}^i, R_{\eta} \epsilon_j)$, $i=1,2$, such that $\epsilon_j< \epsilon_{0}$.  We will denote these balls by $B_j^1$ and $B_j^2$ respectively. 
	We consider the following cases. 
	
	Case 1. There is a subsequence $\epsilon_j\rightarrow 0$, such that $d(x^1_{\epsilon_j},x^2_{\epsilon_j})< \tilde R\epsilon_j$.
	
	This case is ruled out explicitly by Lemma \ref{convergence}, since it would imply that both maxima, $x^1_{\epsilon_j}$ and $x^2_{\epsilon_j}$, are contained in a ball  $B_g(x_{\epsilon_j}^1, \tilde R \epsilon_j)$.

	Case 2.  For each subsequence such that  $\epsilon_j\rightarrow 0$, we have $d(x^1_{\epsilon_j},x^2_{\epsilon_j})\geq \tilde R\epsilon_j$.
	% $\lim_{j\rightarrow \infty} \frac{1}{\epsilon_j} d(x^1_{\epsilon_j},x^2_{\epsilon_j})= \infty$.
	
	In this case we have $B^1_j \cap B^2_j= \varnothing$. And thus, by Lemma \ref{convergence},
	
$$	J_{\epsilon}(u_{\epsilon_j})\geq J_{\epsilon}(u_{\epsilon_j})\Big|_{B_j^1}+J_{\epsilon}(u_{\epsilon_j})\Big|_{B_j^2}>\eta\  m(E)+ \eta\  m(E) =2\eta\  m(E)>m(E)+\delta,$$
	
%(\frac{2p}{p-2})

\noindent  we reach a contradiction to a hypothesis of Theorem 1.2.

%	\item Case 2. $d(x^1_{\epsilon_j},x^2_{\epsilon_j})\rightarrow 0$ as $j\rightarrow \infty$.
	
%		\item Case 3. $d(x^1_{\epsilon_j},x^2_{\epsilon_j})>0$, as $j\rightarrow \infty$.	

	\end{proof}

\section{Proof of Proposition 2.1}

In this section we sketch a proof for the finite dimensional reduction, Proposition 2.1. A detailed proof in a similar situation can
be found in  \cite{Dancer, Micheletti}.

\begin{proof}
	Recall the operator
	$$S_{\epsilon} = \nabla J_{\epsilon} : H_{\epsilon} \rightarrow H_{\epsilon}.$$

and eq. (\ref{Pi})	
	\begin{equation}
	\Pi_{\epsilon,\overline{\xi}}^{\perp}\{ S_{\epsilon} ( V_{\epsilon,\overline{\xi}}+\phi ) \}=0 .
	\end{equation}

We may  rewrite eq. (\ref{Pi})	as
	
	$$0=	\Pi_{\epsilon,\overline{\xi}}^{\perp}\{ S_{\epsilon} ( V_{\epsilon,\overline{\xi}}+\phi ) \}= \Pi_{\epsilon,\overline{\xi}}^{\perp}\{ S_{\epsilon} ( V_{\epsilon,\overline{\xi}} ) +  S'_{\epsilon} ( V_{\epsilon,\overline{\xi}} ) \  \phi+ \bar N_{\epsilon,\bar \xi} (\phi)\}
	=-R_{\epsilon,\bar \xi}+ L_{\epsilon,\bar \xi}(\phi)- N_{\epsilon,\bar \xi}(\phi) .$$
	with the first term being independent of $\phi$:
		$$R_{\epsilon,\bar \xi}:=\Pi_{\epsilon,\overline{\xi}}^{\perp}\{ S_{\epsilon} ( V_{\epsilon,\overline{\xi}} )\}=\Pi^{\perp}_{\epsilon,\bar \xi}\{ i^*_{\epsilon}[f(V_{\epsilon,\bar \xi})]-V_{\epsilon,\bar \xi} \},$$
		
\noindent	the second term, the linear operator:	
		$$L_{\epsilon,\bar \xi}(\phi)=\Pi_{\epsilon,\overline{\xi}}^{\perp}\{   S'_{\epsilon} ( V_{\epsilon,\overline{\xi}} ) \  \phi \}=\Pi^{\perp}_{\epsilon,\bar \xi}\{ \phi -i^*_{\epsilon}[f'(V_{\epsilon,\bar \xi})\phi] \} $$

\noindent			and the last term a  remainder:
		$$N_{\epsilon,\bar \xi}(\phi):=\Pi^{\perp}_{\epsilon,\bar \xi}\{  \bar N_{\epsilon,\bar \xi} (\phi)\}=\Pi^{\perp}_{\epsilon,\bar \xi}\{i^*_{\epsilon}[f(V_{\epsilon,\bar \xi}+\phi)-f(V_{\epsilon,\bar \xi})-f'(V_{\epsilon,\bar \xi})\phi] \}.$$
		
		Hence, eq. (\ref{Pi}) can be written as
	$$L_{\epsilon,\bar \xi}(\phi)=N_{\epsilon,\bar \xi}(\phi)+R_{\epsilon,\bar \xi}.$$

	And then, if $L$ is invertible,  we may turn  eq. (\ref{Pi}), into a fixed point problem, 
	 for the operator $T_{\epsilon,\bar \xi}(\phi):=L^{-1}_{\epsilon,\bar \xi}(N_{\epsilon,\bar \xi}(\phi)+R_{\epsilon,\bar \xi})$. %That is eq. (\ref{2.7}) turns into the fixed point problem: 
	
	%$$\phi=T_{\epsilon,\bar \xi}(\phi).$$
	
	We start by proving that $L_{\epsilon,\bar{\xi}}$ is in fact invertible, for appropriate $\bar \xi$ and $\epsilon$.
	
	\begin{lem}
		\label{L_inversible}
		There exists  $\epsilon_0>0$ and $c>0$, such that for any  $\epsilon \in (0,\epsilon_0)$ and $\bar \xi \in M^K$, $\bar \xi =(\xi_1,\xi_2,...,\xi_K)$, such that 
		$$ \sum_{i,k= 1,i\neq k}^{K} U\left(\frac{\exp_{\xi_i}^{-1}\xi_k}{\epsilon}\right)<\epsilon^2,$$
		we have, 
		
		$$||L_{\epsilon,\bar{\xi}}(\phi)||_{\epsilon}\geq  c ||\phi||_{\epsilon},$$
		
		\noindent for any $\phi \in K^{\perp}_{\epsilon,\bar \xi}$.
		
	\end{lem}

	\begin{proof}
		We will proceed by contradiction. Suppose that there are sequences $\{\epsilon_j\}_{j \in \mathbb{N}}$, $\epsilon_j \rightarrow 0$ and $\{\bar \xi_j\}_{j \in \mathbb{N}}$, $\bar \xi_j =(\xi_{1_j},\xi_{2_j},...,\xi_{K_j})$, such that 
		
		$$ \sum_{i,k= 1,i\neq k}^{K} U\left(\frac{\exp_{\xi_{i_j}}^{-1}\xi_{k_j}}{\epsilon_j}\right)<\epsilon_j^2,$$
		and  $\{\phi_j \}\subset K^{\perp}_{\epsilon,\bar \xi}$, such that 
		$L_{\epsilon_j,\bar{\xi_j}}(\phi_j)=\psi_j$, with $||\phi_j||_{\epsilon_j}=1$ and  $||\psi_j||_{\epsilon_j}\rightarrow 0.$

		Let $\zeta_j:=\Pi_{\epsilon_j,\bar \xi_j}\{ \phi_j -i^*_{\epsilon_j}[f'(V_{\epsilon_j,\bar \xi_j})\phi_j] \}$. Hence,
		
	\begin{equation}\label{phi_j}
	\phi_j -i^*_{\epsilon_j}[f'(V_{\epsilon_j,\bar \xi_j})\phi_j]=\psi_j+\zeta_j.
	\end{equation}

		\noindent	That is, for each $j$, $\psi_j\in K^{\perp}_{\epsilon_j,\bar \xi_j}$ and $ \zeta_j \in K_{\epsilon_j,\bar \xi_j}$. Now, let $u_j:= \phi_j-(\psi_j+\zeta_j).$
		
		We will prove the following contradictory consequences of the existence of such series:
		
		\begin{equation}
		\label{one}
		\frac{1}{\epsilon_j^n} \int_{M} f'(V_{\epsilon_j,\bar {\xi}_j}) u_j^2 \ d\mu_{g}\rightarrow 1,
		\end{equation}
		\noindent and 
		\begin{equation}
		\label{zero}
		\frac{1}{\epsilon_j^n} \int_{M} f'(V_{\epsilon_j,\bar \xi_j}) u_j^2 \  d\mu_{g} \rightarrow 0,	
		\end{equation}

		\noindent this will prove that such  sequences $\{\bar \xi_j\}$, $\{ \phi_j\}$, $\{\epsilon_j\}$ cannot exist.  	
		%	Claim 1 will help us to prove claim 2. And, of course, claims 2 and 3 yield a contradiction.
		We start by proving (\ref{one}).

		First we note that:

\begin{equation}\label{zeta_j}
\Vert\zeta_{j}\Vert_{\epsilon_{j}}\rightarrow 0 \hbox{ as }  j\rightarrow\infty. 
\end{equation}
Since $\zeta_{j}\in K_{\epsilon_{j},\bar{\xi}_j}$ , let $\zeta_{j} :=\displaystyle\sum_{i=1}^{K}\sum_{k=1}^{n}a_{j}^{ki}Z_{\epsilon_{j},\xi_{i_j}}^{k}$. 
For any $h \in \{1,2,\dots, n\}$ , and $l \in \{1,2,\dots,K\}$ we multiply $\psi_{j}+\zeta_j$ (eq. (\ref{phi_j})) by $Z_{\epsilon_{j},\xi_{{l}_j}}^{h}$, and we find

\begin{equation}\label{product}
\begin{array}{lcc}
\displaystyle\sum_{i=1}^{K} \sum_{k=1}^{n}a_{j}^{ki}\langle Z_{\epsilon_{j},\xi_{i_j}}^{k},\ Z_{\epsilon_{j},\xi_{l_j}}^{h}\rangle_{\epsilon_{j}}
&=&\langle \phi_{j}, \ Z_{\epsilon_{j},\xi_{l_j}}^{h} \rangle_{\epsilon_{j}} - 
\langle \io_{\epsilon_{j}}^{*}[ f'(V_{\epsilon_j,\bar {\xi}_j}) \phi_{j}],\ Z_{\epsilon_{j},\xi_{l_j}}^{h}\rangle_{\epsilon_{j}}. \\
%&=&-
%\displaystyle\frac{1}{\epsilon_{j}^{n}}\displaystyle\int_{M} f'(V_{\epsilon_j,\bar {\xi}_j}) \phi_{j}Z_{\epsilon_{j},\xi_{l_j}}^{h} \ d\mu_{g}\\
\end{array}
\end{equation}

%$$=-\frac{1}{\epsilon_{j}^{n}} \int_{M} f'(W_{\epsilon_{j},\xi_{j}})\phi_{j} Z_{\epsilon_{j},\xi_{j}}^{h} \ d\mu_{g}$$
On the other hand, by (\ref{Zeta}),

$$\displaystyle\sum_{i=1}^{K} \sum_{k=1}^{n}a_{j}^{ki}\langle Z_{\epsilon_{j},\xi_{i_j}}^{k},\ Z_{\epsilon_{j},\xi_{l_j}}^{h}\rangle_{\epsilon_{j}}=Ca_{j}^{hl}+o(1),$$

\noindent combining this and (\ref{product}):
\begin{equation}\label{a_j1}
Ca_{j}^{hl}+o(1)=\displaystyle\frac{1}{\epsilon_{j}^{n}}\displaystyle\int_{M}[\epsilon_{j}^{2}\nabla_{g}Z_{\epsilon_{j},\xi_{l_j}}^{h}\nabla_{g}\phi_{j}+
(\ep_j^2 {\bf c} s_g +1) Z_{\epsilon_{j},\xi_{j_{j}}}^{h}\phi_{j}-f'(V_{\epsilon_j,\bar {\xi}_j}) \phi_{j}Z_{\epsilon_{j},\xi_{l_j}}^{h}] \ d\mu_{g}. \\
\end{equation}
Let
\begin{equation}
\tilde{\phi_l}_{j}(z)=
\left\{\begin{array}{lcc}
\phi_{l_j}(\exp_{\xi_{l_j}}(\epsilon_{j}z)\chi_{r}(\epsilon_{j}z)) & \mathrm{i}\mathrm{f}\ z\in B(0,\ r/\epsilon_{j})\ ;\\
0 & \mathrm{o}\mathrm{t}\mathrm{h}\mathrm{e}\mathrm{r}\mathrm{w}\mathrm{i}\mathrm{s}\mathrm{e},
\end{array}\right.
\end{equation}

Then we have that  for some constant $\tilde c$, $\Vert\tilde{\phi_l}_{j}\Vert_{H^{1}(\mathbb{R}^{n})}\leq \tilde c\Vert \tilde \phi_{l_j}\Vert_{\epsilon_{j}}\leq \tilde c$. 
Therefore, we can assume that $\tilde{\phi_l}_{j}$ converges weakly to some $\tilde{\phi}$  in $H^{1}(\mathbb{R}^{n})$ and 
strongly in $L_{loc}^{q}(\mathbb{R}^{n})$ for any $q\in[2,p_n)$. Also note that,

$$| \displaystyle\frac{1}{\epsilon_{j}^{n}} \int_{M} \ep_j^2  s_g  Z_{\epsilon_{j},\xi_{l_j}}^{h}\phi_{j} \  d\mu_{g} |\leq  \ep_j^2  \  c_1 \  | \  \displaystyle \int_{B\left(0,\frac{r}{\epsilon_j}\right)}  \psi^h(z) \chi_r(\epsilon_j z) \tilde \phi_{l_j}(z)\  |g_{\xi_{{l}_j}(\epsilon_{j} z)}|^{1/2}\  dz|$$
$$= \ \ep_j^2 \  c_1\left(\int_{\re^n} \psi^h \ \tilde \phi \ dz +o(1)\right)\leq  \ \ep_j^2 \  c_1\left(\int_{\re^n} (\psi^h)^2 \ dz \right)^{1/2}\left(\int_{\re^n} \tilde \phi^2 \ dz \right)^{1/2}+ o(\ep_j^2)$$
\begin{equation}
\label{sgeps}
 \leq c_1 \  c_2 \  \ep_j^2 \ ||\tilde \phi||_{L^2(\re^n)}  + o(\ep_j^2)\leq \ c_1 \ c_2 \ c_3 \  \ep_j^2 + o(\ep_j^2) =  o(\ep_j),
 \end{equation}

\noindent where $c_1$ is an upper bound for $s_g$, $c_2$ for $|| \nabla U||_{L^2(\re^n)}\ $ and $ c_3 \ $ for  $|| \tilde \phi||_{L^2(\re^n)}\ $.

Then we have, by eqs. (\ref{a_j1}) and (\ref{sgeps})
\begin{equation}\label{a_j2}
\begin{array}{lcc}
&Ca_{j}^{hl}+o(1)&=\displaystyle\frac{1}{\epsilon_{j}^{n}}\displaystyle\int_{M}[\epsilon_{j}^{2}\nabla_{g}Z_{\epsilon_{j},\xi_{l_j}}^{h}\nabla_{g}\phi_{j}+
(\ep_j^2 {\bf c}s_g +1) Z_{\epsilon_{j},\xi_{l_j}}^{h}\phi_{j}-f'(V_{\epsilon_j,\bar {\xi}_j})\phi_{j}Z_{\epsilon_{j},\xi_{l_j}}^{h}] \ d\mu_{g} \\
&=&\displaystyle\frac{1}{\epsilon_{j}^{n}}\displaystyle\int_{M}[\epsilon_{j}^{2}\nabla_{g}Z_{\epsilon_{j},\xi_{l_j}}^{h}\nabla_{g}\phi_{j}+
 Z_{\epsilon_{j},\xi_{l_j}}^{h}\phi_{j}-f'(V_{\epsilon_j,\bar {\xi}_j})\phi_{j}Z_{\epsilon_{j},\xi_{l_j}}^{h}] \ d\mu_{g} +o(\epsilon_{j}) \\
%&\displaystyle \int_{B(0,\frac{r}{\epsilon_j})}[\epsilon_{j}^{2}\nabla_{g}Z_{\epsilon_{j},\xi_{j}}^{h}\nabla_{g}\phi_{j}+
%(\ep_j^2 s_g +1) Z_{\epsilon_{j},\xi_{j}}^{h}\phi_{j}-f'(V_{\epsilon_j,\bar {\xi}_j})\phi_{j}Z_{\epsilon_{j},\xi_{j}}^{h}] \ dz \\
&=&\displaystyle\int_{\mathbb{R}^{n}}(\nabla\psi^{h}\nabla\tilde{\phi}+\psi^{h}\tilde{\phi}-f'(U)\psi^{h}\tilde{\phi})\ dz  +o(1)=o(1).\\

\end{array}
\end{equation}

From $(\ref{a_j2})$, we get that $a_{j}^{hl}\rightarrow 0$ for any $h=1, \cdots, n$, and any $l=1, \cdots, K$ and then $(\ref{zeta_j})$ follows.
We are ready to prove (\ref{one}).

Recall that $u_j= \phi_j-(\psi_j+\zeta_j)$, since $\Vert\phi_{j}\Vert_{\epsilon_{j}}=1, \Vert\psi_{j}\Vert_{\epsilon_{j}}\rightarrow 0$ and $\Vert\zeta_{j}\Vert_{\epsilon_{j}}\rightarrow 0$ 
then
\begin{equation}
\Vert u_{j}\Vert_{\epsilon_{j}}\rightarrow 1.
\end{equation}
Moreover, by (\ref{phi_j}) $u_{j}=\io_{\epsilon_{j}}^{*}[f'(W_{\epsilon_{j},\xi_{j}})\phi_{j}]$, hence, by (\ref{Adjoint}), it satisfies weakly
\begin{equation}\label{u_j solves}
-\epsilon_{j}^{2}\triangle_{g} u_{j}+(\ep_j^2 {\bf c} s_g+1)u_{j}=f'(V_{\epsilon_j,\bar {\xi}_j})u_{j}  
+f'(V_{\epsilon_j,\bar {\xi}_j})(\psi_{j}+\zeta_{j}) \hbox{ in } M.
\end{equation} 
%%%%%%%%%%%%%%%%%%%%%%%%%%%%%%%%%%%%%%%%%%%%%%%%%%%%%%%%%%%%%%%%%%%%%%%%%%%%%%%%%%%%%
Multiplying (\ref{u_j solves}) by $u_{j}$, and integrating over $M$,
\begin{equation}
\label{85}
\displaystyle \Vert u_{j}\Vert_{\epsilon_{j}}^{2}=\frac{1}{\epsilon_{j}^{n}}\int_{M}f'(V_{\epsilon_j,\bar {\xi}_j})u_{j}^{2} \ d\mu_{g}+
\frac{1}{\epsilon_{j}^{n}}\int_{M}f'(V_{\epsilon_j,\bar {\xi}_j})(\psi_{j}+\zeta_{j})u_{j} \ d\mu_{g}
\end{equation}

By H\"{o}lder's inequality and eq. (23) we can find  eq. (\ref{one}):

$$|\displaystyle \frac{1}{\epsilon_{j}^{n}}\int_{M}f'(V_{\epsilon_j,\bar {\xi}_j})(\psi_{j}+\zeta_{j})u_{j} \ d\mu_{g}|$$
$$\leq \left(\displaystyle \frac{1}{\epsilon_{j}^{n}}\int_{M}(f'(V_{\epsilon_j,\bar {\xi}_j})\ u_{j})^2 \ d\mu_{g}\right)^{\frac{1}{2}} \left(\displaystyle \frac{1}{\epsilon_{j}^{n}}\int_{M}(\psi_{j}+\zeta_{j})^2\ d\mu_{g}\right)^{\frac{1}{2}}$$
%$$\leq \vert f'(V_{\epsilon_j,\bar {\xi}_j})\vert_{\frac{n}{2},\epsilon_j} \ \vert(\psi_{j}+\zeta_{j})\vert_{\frac{2n}{n-2},\epsilon_j}\ \vert u_{j}\vert_{\frac{2n}{n-2},\epsilon_j} $$

$$\leq c  \  ||u_{j}||_{\epsilon_j} \ ||\psi_{j}+\zeta_{j}||_{\epsilon_j} = o(1),$$

\noindent   since $\Vert\psi_{j}\Vert_{\epsilon_{j}}\rightarrow 0$, $\Vert\zeta_{j}\Vert_{\epsilon_{j}}\rightarrow 0$, and $\Vert u_{j}\Vert_{\epsilon_{j}}\rightarrow 1$ as $j\rightarrow\infty$. We conclude from eq. (\ref{85}) that 	$\frac{1}{\epsilon_j^n} \int_{M} f'(V_{\epsilon_j,\bar {\xi}_j}) u_j^2 \ d\mu_{g}\rightarrow 1$.

Finally, we prove eq. (\ref{zero}).
%%%%%%%%%%%%%%%%%%%%%%%%%%%%%%%%%%%%%%%%%%%%%%%%%%%%%%%%%%%%%%%%%%%%%%%%%%%%%%%%%%%%%%%%%%%%%%%%%%%%%%%%%%%%%%%%%%%
%%%%%%%%%%%%%%%%%%%%%%%%%%%%%%%%%%%%%%%%%%%%%%%%%%%%%%%%%%%%%%%%%%%%%%%%%%%%%%%%%%%%%%%%%%%%%%%%%%%%%%%%%%%%%%%%%%%
%Note that $u_{j}$ is compactly supported in $B_{g}(\xi_{j},\ r)$ .

Given $l \in \{1,\cdots , K\}$, we define
$$
\tilde{u}_{l_j}=u_{j}\left(\exp_{\xi_{l_j}}(\epsilon_{j}z)\right)  \ \chi_{r}\left(\exp_{\xi_{{l}_j}}(\epsilon_{j}z)\right), \  z\in \re^n\ .
$$
Note that $\Vert\tilde{u}_{l_j}\Vert_{H^{1}(\mathbb{R}^{n})}^{2}\leq c\Vert {u}_{j}\Vert_{\epsilon_{j}}^{2}\leq c$.
Then, up to a subsequence, $\tilde{u}_{l_j}\rightarrow\tilde{u_l}$ weakly in $H^{1}(\mathbb{R}^{n})$ and 
strongly in $L_{loc}^{q}(\mathbb{R}^{n})$ for any $q\in[2,\ p_n)$, for some $\tilde u_l \in H^{1}(\mathbb{R}^{n}) $ . 

We now claim that $\tilde{u_l}$ solves weakly the problem
\begin{equation}\label{u tilde}
-\triangle\tilde{u_l}+\tilde{u_l}=f'(U)\tilde{u_l} \hspace{0.5cm} \hbox{ in }\hspace{0.5cm}  \mathbb{R}^{n}.
\end{equation}

Let $\varphi \in \C_0(\re^n)$. Set $\varphi_j(x):=\varphi\left(\frac{\exp^{-1}_{\xi_{{l}_j}}(x)}{\epsilon_j}\right)  \ \chi_{r}\left(\exp^{-1}_{\xi_{{l}_j}}(x)\right)$, for  $x$ in $B(\xi_{l_j},\epsilon_j \  R) \subset M$. For $R$ big enough such that $\mathit{supp} \  \varphi \subset B(0,R)$ and $j$ big enough such that $B(\xi_{l_j},\epsilon_j \  R) \subset B(\xi_{l_j},r)$.

Multiplying  $(\ref{u_j solves})$ by $\varphi_j$ and integrating over $M$, 

$$\frac{1}{\epsilon^n} \int_M \left( \epsilon_j^2 \nabla_g u_j \ \nabla_g \varphi_j + (1+ s_g {\bf c}\epsilon_j^2) \ u_j \varphi_j \right) d {\mu_g}$$

$$=\frac{1}{\epsilon_{j}^{n}}\int_{M}f'(V_{\epsilon_j,\bar {\xi}_j}) \ u_{j} \  \varphi_j \ d\mu_{g}+
\frac{1}{\epsilon_{j}^{n}}\int_{M} \ f'(V_{\epsilon_j,\bar {\xi}_j})(\psi_{j}+\zeta_{j}) \  \varphi_j  \ d\mu_{g}.$$

We may rewrite this equation in $\re^n$ by setting $x=\exp_{\xi_{{l}_j}}(\epsilon_j \ z)$:

$$ \int_{B(0,R)} \left( \sum_{s,t=1}^{n} g^{st}_{\xi_{{l}_j}}(\epsilon_j z) \frac{\partial \tilde u_l{_j}}{\partial {z_s}}  \  \frac{\partial \varphi}{\partial {z_t}}    + (1+ s_g {\bf c}\epsilon_j^2) \tilde u_{l_j} \varphi \right) |g_{\xi_{{l}_j}}(\epsilon_{j} z)|^{1/2} dz$$

$$= \int_{B(0,R)} f'\left( U(z) \chi_r(\epsilon_j z)+ \sum_{i\neq l} U\left(\frac{\exp^{-1}_{\xi_{{l}_j}} \exp_{\xi_{{i}_j}}(\epsilon_{j} z)}{\epsilon_j}\right) \chi_r \left(\exp^{-1}_{\xi_{{l}_j}} \exp_{\xi_{{i}_j}}(\epsilon_{j} z)\right)  \right)$$ 
$$\tilde u_{l_j} \  \varphi \  |g_{\xi_{{l}_j}(\epsilon_{j} z)}|^{1/2} dz$$

\begin{equation}
\label{int}
+ \int_{B(0,R)} f'\left( U(z) \chi_r(\epsilon_j z)+ \sum_{i\neq l} U\left(\frac{\exp^{-1}_{\xi_{{l}_j}} \exp_{\xi_{{i}_j}}(\epsilon_{j} z)}{\epsilon_j}\right) \chi_r \left(\exp^{-1}_{\xi_{{l}_j}} \exp_{\xi_{{i}_j}}(\epsilon_{j} z)\right)  \right)
\end{equation}
$$ (\tilde \psi_j +\tilde \zeta_j) \  \varphi \  |g_{\xi_{{l}_j}} (\epsilon_{j} z)|^{1/2} dz,$$

\noindent where $\tilde \psi_{j}(z):=\psi_{j} (\exp_{\xi_{{l}_j}}(\epsilon_j \ z))$ and $\tilde \zeta_{j}(z):=\zeta_{j} (\exp_{\xi_{{l}_j}}(\epsilon_j \ z))$ for $z \in B(0, R/\epsilon_j)$.

Note that 
$$ \int_{B(0,R)} s_g \epsilon_j^2 \tilde u_{l_j} \varphi |g_{\xi_{{l}_j}}(\epsilon_{j} z)|^{1/2} dz\leq c \ \epsilon_j^2  \int_{B(0,R)} \tilde u_{l_j} \varphi |g_{\xi_{{l}_j}}(\epsilon_{j} z)|^{1/2} dz $$
$$\leq c \ \epsilon_j^2    \left(\int_{B(0,R)} \tilde u_{l_j}^2  dz\right)^{1/2}  \left(\int_{B(0,R)} \varphi^2 |g_{\xi_{{l}_j}}(\epsilon_{j} z)| dz\right)^{1/2} \leq c \  \epsilon_j^2    ||\tilde u_{l_j}|||_{H^1(\re^n)} c_2=o(\epsilon_{j}).$$
\noindent with $c$ an upper bound for $s_g$, $c_2^2$ an upper bound for $\int_{B(0,R)} \varphi^2 |g_{\xi_{{l}_j}(\epsilon_{j} z)}| dz$. Recall also that $u_{l_j}$ is bounded independently of $j$ in $H^1(\re^n)$.

Hence, taking the limit as $\epsilon_j \rightarrow 0$, in (\ref{int})

\begin{equation}
\label{weaklysolves}
\int_{\re^n} \left( \sum_{s,t=1}^{n} \delta_{s,t}  \frac{\partial \tilde u_l}{\partial {z_s}}  \  \frac{\partial \varphi}{\partial {z_t}}    + \tilde u_{l} \varphi \right) dz = \int_{\re^n} f'\left( U(z)  \right) \tilde u_{l} \  \varphi \  dz,
\end{equation} 
 \noindent since $\tilde \psi_j, \tilde \zeta_j \rightarrow 0$ strongly in $H^1(\re^n)$. Eq. (\ref{weaklysolves}) for each $\varphi \in C_0^{\infty}(\re^n)$, implies the claim that $\tilde u_l$ solves weakly eq. $(\ref{u tilde})$ in $\re^n$.

We now claim that for any $h \in \{1,2,...n\}$, $\tilde u_l$ satisfies also
\begin{equation}
\label{productu}
\int_{\mathbb{R}^{n}}\left(\nabla\psi^{h}\nabla\tilde{u_l}+\psi^{h}\tilde{u_l}\right) \ dz=0.
\end{equation}

 To prove (\ref{productu}) we compute
\begin{equation}\label{lange}
|\langle Z_{\epsilon_{j},\xi_{l_j}}^{h},\ u_{j}\rangle _{\epsilon_{j}}|=| \langle Z_{\epsilon_{j},\xi_{l_j}}^{h},\phi_{j} -\psi_j-\zeta_j \rangle_{\epsilon_{j}}|=|
\langle Z_{\epsilon_{j},\xi_{l_j}}^{h},\ \zeta_{j}\rangle_{\epsilon_{j}}|
\leq\Vert Z_{\epsilon_{j},\xi_{l_j}}^{h}\Vert_{\epsilon_{j}} \ \ \Vert\zeta_{j}\Vert_{\epsilon_{j}}=o(1), 
\end{equation}

\noindent since $\phi_{j}, \psi_{j}\in K_{\epsilon_{j},\bar \xi_{j}}^{\perp}$ and eq. (\ref{zeta_j}). On the other hand, we have

$$
\langle Z_{\epsilon_{j},\xi_{l_j}}^{h},\ u_{j}\rangle_{\epsilon_{j}}=\frac{1}{\epsilon_{j}^{n}}\int_{M}
[\epsilon_{j}^{2}\nabla_{g}Z_{\epsilon_{j},\xi_{l_j}}^{h}\nabla_{g}u_{j}+ (\ep_j^2 {\bf c}s_g +1)
Z_{\epsilon_{j},\xi_{l_j}}^{h}u_{j}] \ d\mu_{g}.
$$
Of course, by H\"{o}lder's inequality and eq. (23): 

$$\left|  \frac{1}{\epsilon_{j}^{n}} \int_{M} \epsilon_j^2 {\bf c} s_g Z_{\epsilon_{j},\xi_{l_j}}^{h}u_{j} \ d\mu_{g}\right| \leq c \  \epsilon_j^2 \left|  \frac{1}{\epsilon_{j}^{n}} \int_{M}  Z_{\epsilon_{j},\xi_{l_j}}^{h}u_{j} \ d\mu_{g}\right|$$

$$\leq c \  \epsilon_j^2\left(\displaystyle \frac{1}{\epsilon_{j}^{n}}\int_{M}(Z_{\epsilon_{j},\xi_{l_j}}^{h})^2 \ d\mu_{g}\right)^{\frac{1}{2}} \left(\displaystyle \frac{1}{\epsilon_{j}^{n}}\int_{M}(u_{j})^2\ d\mu_{g}\right)^{\frac{1}{2}}$$
%$$\leq \vert f'(V_{\epsilon_j,\bar {\xi}_j})\vert_{\frac{n}{2},\epsilon_j} \ \vert(\psi_{j}+\zeta_{j})\vert_{\frac{2n}{n-2},\epsilon_j}\ \vert u_{j}\vert_{\frac{2n}{n-2},\epsilon_j} $$
$$
\leq c \  \epsilon_j^2  \left( \int_{B(0,r/\epsilon_{j})}
\left(\psi^{h}(z) \ \chi_{r}(\epsilon_{j}z)\right)^{2}  |g_{\xi_{{l}_j}}(\epsilon_{j} z)|^{1/2} dz\right)^{\frac{1}{2}} \left( \int_{M} u_{j}^2\ d\mu_{g}\right)^{\frac{1}{2}}
$$

$$\leq c \  \epsilon_j^2  \ \left( \int_{\re^n}|\nabla U|^2 dz + o(1)\right)^{1/2} \ ||u_{j}||_{\epsilon_j} = o(\epsilon_{j}),$$

\noindent   since $\psi^h (z)= \frac{\partial U}{\partial {z_h}} (z)$, and $\Vert u_{j}\Vert_{\epsilon_{j}}\rightarrow 1$ as $j\rightarrow\infty$. Then 

$$\langle Z_{\epsilon_{j},\xi_{l_j}}^{h},\ u_{j}\rangle_{\epsilon_{j}}=\frac{1}{\epsilon_{j}^{n}}\int_{M}
[ \ \epsilon_{j}^{2}\nabla_{g}Z_{\epsilon_{j},\xi_{l_j}}^{h}\nabla_{g}u_{j} + 
Z_{\epsilon_{j},\xi_{l_j}}^{h}u_{j} \ ] \ d\mu_{g}+o(\epsilon_{j})
$$
$$
= \int_{B(0,r/\epsilon_{j})}[\sum_{s,t=1}^{n}g_{\xi_{l_j}}^{st}(\epsilon_{j}z)\frac{\partial}{\partial z_{s}}
\left(\psi^{h}(z)\chi_{r}(\epsilon_{j}z)\right)
\frac{\partial}{\partial z_{t}}\left(\tilde{u}_{l_j}(z)\right)$$
$$
+ \ \psi^{h}(z)\chi_{r}(\epsilon_{j}z)\tilde{u}_{l_j}(z) \  ]\  |g_{\xi_{l_j}}(\epsilon_{j}z)|^{\frac{1}{2}} \ dz + o(\epsilon_{j})
$$

\begin{equation}\label{tilde}
=\displaystyle \int_{\mathbb{R}^{n}}\left(\nabla\psi^{h}\nabla\tilde{u}+\psi^{h}\tilde{u}\right) \ dz+o(1).  
\end{equation}
From $(\ref{lange})$ and $(\ref{tilde})$ we prove the claim of eq. (\ref{productu}).

Therefore, by $(\ref{u tilde})$ and $(\ref{productu})$ it follows that $\tilde{u}=0$. %$\tilde{u}_{j}\rightarrow 0$ weakly 
%in $H^{1}(\mathbb{R}^{n})$ and strongly in$L_{loc}^{q}(\mathbb{R}^{n})$ for any $q\in[2,p_n)$. 

We now prove eq. (\ref{zero}). We will estimate	$\frac{1}{\epsilon_j^n} \int_{M} f'(V_{\epsilon_j,\bar \xi_j}) u_j^2 \  d\mu_{g}$ by partitioning $M$ in various subsets. First we will make estimates in small neighborhoods around each $\xi_{{l_j}}$, $l\in \{1, 2, ..., K\}$, using the fact that $\tilde u_l=0$. Then we will make estimates in the complement of these neighborhoods using  the hypothesis that 	$$ \sum_{i,k= 1,i\neq k}^{K} U\left(\frac{\exp_{\xi_{i_j}}^{-1}\xi_{k_j}}{\epsilon_j}\right)<\epsilon_j^2.$$

Let $ R_j=\frac{1}{2} \min\{d_g\left(\xi_{{l}_j},\xi_{{m}_j}\right), l\neq m\} $. Let $\tilde M=\bigcup\limits_{l=1}^K B_g(\xi_{l_j}, R_j) $. Then

\begin{equation}
\label{all}
\frac{1}{\epsilon_{j}^{n}}\int_{M}f'(V_{\epsilon_{j}, \bar\xi_{j}})u_{j}^{2} d\mu_{g}=\frac{1}{\epsilon_{j}^{n}}\sum^K_{l=1}\int_{B_g(\xi_{{l_j}}, R_j)}f'(V_{\epsilon_{j}, \bar\xi_{j}})u_{j}^{2} d\mu_{g}
+\frac{1}{\epsilon_{j}^{n}}\int_{M\setminus \tilde M} f'(V_{\epsilon_{j}, \bar\xi_{j}})u_{j}^{2} d\mu_{g}.
\end{equation}

Now, on one hand, for each $l$, since $\tilde u_l=0$,
$$\frac{1}{\epsilon_{j}^{n}} \int_{B_g(\xi_{{l_j}},  R_j)}f'(V_{\epsilon_{j}, \bar\xi_{j}})u_{j}^{2} d\mu_{g}$$
$$=\int_{{B(0, \ \epsilon_j  R_{j})}}  f'\left(U(z) \chi_{r}(\epsilon_{j} z) + \sum_{i\neq l}^K U\left(\frac{\exp^{-1}_{\xi_{{l}_j}} \exp_{\xi_{{i}_j}}(\epsilon_{j} z)}{\epsilon_j}\right) \chi_r \left(\exp^{-1}_{\xi_{{l}_j}} \exp_{\xi_{{i}_j}}(\epsilon_{j} z)\right)  \right)$$
$$\tilde u_{l_j}^2(z) \  |g_{\xi_{{l}_j}}(\epsilon_{j} z)|^{1/2} dz$$

%$$=\frac{1}{\epsilon_{j}^{n}}
%\int_{B_{g}(\xi_{j},r)}f'(U_{\epsilon_{j}})(exp_{\xi_{l_j}}^{-1}(x))u_{j}^{2}(x)dx
%$$
%$$
%\leq C\int_{B(0,r/\epsilon_{j})}f'(U(z))\tilde{u}_{j}^{2}(z) \ dz
%$$
\begin{equation}
\label{neighborhoods}
= o(1).
\end{equation}

On the other hand, by H\"older's inequality 
$$\frac{1}{\epsilon_{j}^{n}}\int_{M\setminus \tilde M} f'(V_{\epsilon_{j}, \bar\xi_{j}})u_{j}^{2} d\mu_{g}$$
$$\leq \left(\frac{1}{\epsilon_{j}^{n}}\int_{M\setminus \tilde M} \left(f'(V_{\epsilon_{j}, \bar\xi_{j}})\right)^{n/2} d\mu_{g}\right)^{2/n} \left(\frac{1}{\epsilon_{j}^{n}}\int_{M\setminus \tilde M}u_{j}^{\frac{2n}{n-2}} d\mu_{g}\right)^{\frac{n-2}{n}}$$
$$\leq c_1 \left(\frac{1}{\epsilon_{j}^{n}}\int_{M\setminus \tilde M} \left((p-1)\sum^K_{l=1} W_{\epsilon_{j}, \xi_{l_j}}^{(p-2)}\right)^{\frac{n}{2}} d\mu_{g}\right)^{2/n} \  ||u_j||^2_{\epsilon_{j}} $$

$$\leq c_2 \left(\frac{1}{\epsilon_{j}^{n}}\int_{M \setminus \tilde M}  \left(  \sum^K_{l=1}U^{(p-2)} \left( \frac{\exp_{\xi_{l_j}}^{-1}(x)}{\epsilon_j}\right) \chi_{r}^{(p-2)} \left( \frac{\exp_{\xi_{l_j}}^{-1}(x)}{\epsilon_j}\right) \right)^{ \frac{n}{2}} d\mu_{g}\right)^{2/n}  $$

%$$\leq c \sum^K_{l=1}\left(\frac{1}{\epsilon_{j}^{n}}\int_{B_g(\xi_{{l}_j},r) \setminus \tilde M} U^{\frac{(p-2)n}{2}} \left( \frac{\exp_{\xi_{l_j}}^{-1}(x)}{\epsilon_j}\right)d\mu_{g}\right)^{2/n} \   $$
$$\leq c_2 \frac{1}{\epsilon_{j}^{2}}\sum^K_{l=1}\left(\int_{B_g(\xi_{{l}_j},r) \setminus \tilde M} U^{\frac{(p-2)n}{2}} \left( \frac{\exp_{\xi_{l_j}}^{-1}(x)}{\epsilon_j}\right)d\mu_{g}\right)^{2/n} \  $$

\begin{equation}
\label{complement}
\leq c_2 \frac{1}{\epsilon_j^2 }e^{-(p-2)\frac{R_j}{ \epsilon_{j}}} \sum^K_{l=1}\left(\int_{B_g(\xi_{{l}_j},r) \setminus \tilde M} d\mu_{g}\right)^{2/n}\leq c_3 \frac{1}{\epsilon_j^2 }e^{-(p-2)\frac{R_j}{ \epsilon_{j}}} = o(1).
\end{equation}

Eqs. (\ref{all}),  (\ref{neighborhoods}) and  (\ref{complement}) prove (\ref{zero}), which contradicts (\ref{one}) .

\end{proof} 

Next we  study an estimate for the term $ R_{\epsilon, \bar \xi}=\Pi^{\perp}_{\epsilon,\bar \xi}\{ i^*_{\epsilon}[f(V_{\epsilon,\bar \xi})]-V_{\epsilon,\bar \xi} \}$.

\begin{lem}
	\label{estimateR}
 There exist $\rho_0>0$, $\ep_0>0$, $c>0$ and  $\sigma>0$ such that for any $\rho\in(0,\rho_0)$, $\ep\in(0,\ep_0)$
 and $\bar \xi\in D_{\epsilon,\rho}^{{k_0}}$, it holds 
 \begin{equation}
 \Vert R_{\epsilon, \bar \xi}\Vert_\ep \leq c \Big(\ep^2 + \sum_{i\neq j}e^{-\frac{1+\sigma}{2}\frac{d_g(\xi_i,\xi_j)}{\ep}}\Big).
 \end{equation}
\end{lem}

\begin{dem}
	
	Let $Y_{\epsilon ,\xi} = \epsilon^{2}\triangle_{g}W_{\epsilon ,\xi} +(\ep^2 {\bf c}s_g +1)W_{\epsilon ,\xi}$, so that by (\ref{Adjoint}): 
	$W_{\epsilon ,\xi} = \io_{\epsilon}^{*}(Y_{\epsilon ,\xi} )$. 
 %Let $Z_{\ep,\overline{\xi}}$ be given by $V_{\ep,\overline{\xi}}=\io_\ep ^*(Z_{\ep,\overline{\xi}})$. Then it satisfies
 Hence,
 % \begin{equation}
 %-\ep^2\Delta_g W_{\ep,{\xi_i}} + (1+ \epsilon^2 s_g) W_{\ep,{\xi_i}} = Y_{\ep,{\xi_i}} \hspace{0.5cm}\hbox{ on }M.
 %\end{equation}
if  $\ Y_{\epsilon,\overline{\xi}}:=\sum_{i=1}^{k_0}Y_{\epsilon,\xi_{i}},$ we have
 
 \begin{equation}
  -\ep^2\Delta_g V_{\ep,\overline{\xi}} + (1+ \epsilon^2 {\bf c}s_g) V_{\ep,\overline{\xi}} = Y_{\ep,\overline{\xi}} \hspace{0.5cm}\hbox{ on }M,
 \end{equation}
%\begin{equation}

%\end{equation}

\noindent that is, 	$V_{\epsilon ,\bar \xi} = \io_{\epsilon}^{*}(Y_{\epsilon ,\bar \xi} )$. Then, using the estimate in  (\ref{norma_ep}): 
\begin{equation}
\begin{array}{lcc}
  \Vert R_{\epsilon, \bar \xi}\Vert_\ep = \Vert \io_\ep^*(f(V_{\ep, \bar\xi})) - V_{\ep, \bar\xi}) \Vert_\ep \leq C \mid f(V_{\ep, \bar\xi}) - Y_{\ep, \bar\xi}\mid_{p',\ep} \\
 \leq C \mid \Big(\sum_{i=1}^{k_0} W_{\ep, \xi_i}\Big)^{p-1} - \sum_{i=1}^{k_0} W_{\ep,\xi_i}^{p-1} \mid_{p',\ep} + 
 \mid \sum_{i=1}^{k_0} W_{\ep,\xi_i}^{p-1}  - Y_{\ep,\bar \xi_i}\mid_{p',\ep},  \\
 \end{array}
\end{equation} 
\noindent for some $C>0$. On one hand, by arguing as in Lemma $3.3$ in \cite{Dancer}, for some $\sigma>0$, we get
\begin{equation}\label{Orden grande}
  \mid \Big(\sum_{i=1}^{k_0} W_{\ep, \xi_i}\Big)^{p-1} - \sum_{i=1}^{k_0} W_{\ep,\xi_i}^{p-1} \mid_{p',\ep}
  %\leq \frac{c}{\epsilon^n} \sum_{i,j=1}^{k_0} \int_M \left( W_{\ep, \xi_i}^{p-2} W_{\ep, \xi_j}\right)^{p'} d\mu_g =
  =o\Big(\sum_{i\neq j}e^{-\frac{1+\sigma}{2}\frac{d_g(\xi_i,\xi_j)}{\ep}}\Big)
\end{equation}
\noindent On the other hand,

%$$Y_{\ep,\bar \xi_i} =\sum_{i=1}^{k_0}Y_{\epsilon,\xi_{i}}Y_{\ep, \xi_i}$$

$$\mid \sum_{i=1}^{k_0} W_{\ep,\xi_i}^{p-1}  - Y_{\ep,\bar \xi_i}\mid_{p',\ep}=\mid \sum_{i=1}^{k_0} \left(W_{\ep,\xi_i}^{p-1}  - Y_{\ep, \xi_i}\right)\mid_{p',\ep}  $$
\begin{equation}\label{orden ep2}
\leq \sum_{i=1}^{k_0} \mid  W_{\ep,\xi_i}^{p-1}  - Y_{\ep, \xi_i}\mid_{p',\ep}  %=o(\ep^2).
\end{equation}

Let $\tilde Y_{\ep, \xi}(z)= Y_{\ep, \xi}(\exp_{\xi} (z))$ for $z \in B(0,r)$, then
$$\tilde Y_{\ep, \xi}= -\ep^2\Delta_g W_{\ep,{\xi}} + (1+ \epsilon^2 {\bf c}s_g) W_{\ep,{\xi}}= -\ep^2\Delta_g (U_{\epsilon} \chi_r) + (1+ \epsilon^2 {\bf c}s_g) U_{\epsilon} \chi_r $$

$$=-\ep^2 \ \chi_r\  \Delta U_{\epsilon} +  U_{\epsilon} \chi_r  + \epsilon^2 {\bf c}s_g U_{\epsilon} \chi_r -\ep^2 U_{\epsilon} \Delta \chi_r -2 \epsilon^2 \langle \nabla U_{\epsilon}, \nabla \chi_r\rangle $$
$$+\epsilon^2(g_{\xi}^{ij}-\delta_{i,j}) \partial_{ij}(U_{\epsilon} \chi_r)- \epsilon^2 g_{\xi}^{ij} \Gamma^k_{ij} \partial_{k}(U_{\epsilon} \chi_r)$$
$$=\left(U_{\epsilon} ^{p-1} \chi_r -\ep^2 U_{\epsilon}  \Delta \ \chi_r -2 \epsilon^2 \langle \nabla U_{\epsilon}, \nabla \chi_r\rangle +\epsilon^2(g_{\xi}^{ij}-\delta_{i,j}) \partial_{ij}(U_{\epsilon} \chi_r)- \epsilon^2 g_{\xi}^{ij} \Gamma^k_{ij} \partial_{k}(U_{\epsilon} \chi_r)\right) $$
$$+ \left(\epsilon^2 {\bf c} s_g U_{\epsilon} \chi_r\right)$$

Then 
$$ \left(\frac{1}{\epsilon^n}\int_M ( W_{\ep,\xi}^{p-1}  - Y_{\ep, \xi})^{p'} d\mu_g\right)^{\frac{1}{p'}}\ =\left(\frac{1}{\epsilon^n}\int_{B_{(0,r)}} \left( (U_{\epsilon} (z)\chi_r (z))^{p-1}  - \tilde Y_{\ep, \xi(z)}\right)^{p'} \ |g_{\xi}(z)|\  dz \right)^{\frac{1}{p'}}  $$
$$\leq c \left( \frac{1}{\epsilon^n} \int_{B_{(0,r)}}\left(U_{\epsilon} ^{p-1} ( \chi_r^{p-1}-\chi_r )\right)^{p'}  dz\right)^{\frac{1}{p'}} 
+c \ep^2 \left( \frac{1}{\epsilon^n} \int_{B_{(0,r)}}( U_{\epsilon}  \Delta \ \chi_r)^{p'}dz \right)^{\frac{1}{p'}}$$
$$ +c \epsilon^2  \left( \frac{1}{\epsilon^n} \int_{B_{(0,r)}}\left(\langle \nabla U_{\epsilon}, \nabla \chi_r\rangle\right)^{p'} dz \right)^{\frac{1}{p'}}$$ 

$$+c \epsilon^2 \left( \frac{1}{\epsilon^n} \int_{B_{(0,r)}}\left((g_{\xi}^{ij}-\delta_{i,j}) \partial_{ij}(U_{\epsilon} \chi_r) \right)^{p'} dz \right)^{\frac{1}{p'}}+c \epsilon^2 \left( \frac{1}{\epsilon^n} \int_{B_{(0,r)}} \left(  g_{\xi}^{ij} \Gamma^k_{ij} \partial_{k}(U_{\epsilon} \chi_r)\right)^{p'}  dz\right)^{\frac{1}{p'}}$$

$$+c \epsilon^2  \left(\frac{1}{\epsilon^n} \int_{B_{(0,r)}}\left(  s_g U_{\epsilon} \chi_r\right)^{p'} dz \right)^{\frac{1}{p'}}$$

\noindent by  Lemma 3.3 in \cite{Micheletti}, the first five terms in the last inequality are $o(\epsilon^2)$. Meanwhile, for the last term we have:

$$ \left(\frac{1}{\epsilon^n} \int_{B_{(0,r)}}\left( \epsilon^2 s_g U_{\epsilon} \chi_r\right)^{p'} dz \right)^{\frac{1}{p'}}
\leq c_1 \ \epsilon^2 \left(\frac{1}{\epsilon^n} \int_{B_{(0,r)}}   U_{\epsilon}^{p'} \chi_r^{p'} dz \right)^{\frac{1}{p'}}$$
$$\leq c_1 \ \epsilon^2 \left( \int_{B_{(0,\frac{r}{\epsilon})}}   U^{p'}  dz \right)^{\frac{1}{p'}}\leq c_2 \  \epsilon^2.$$

Thus, eq. (\ref{orden ep2}) turns into 

\begin{equation}\label{orden ep3}
\mid \sum_{i=1}^{k_0} W_{\ep,\xi_i}^{p-1}  - Y_{\ep,\bar \xi_i}\mid_{p',\ep}\leq \sum_{i=1}^{k_0} \mid  W_{\ep,\xi_i}^{p-1}  - Y_{\ep, \xi_i}\mid_{p',\ep}  \leq c \ \ep^2.
\end{equation}
\noindent for some $c>0$. Eqs. (\ref{Orden grande}) and (\ref{orden ep3}) imply the estimate of the Lemma.

\end{dem}  

\end{proof} %\hfill$\square $ 

 As stated above, in order to solve eq. $(\ref{Pi})$ we need to find a fixed point for the operator 
 $T_{\ep,\overline{\xi}}:K_{\epsilon,\overline{\xi}}^{\perp} \to K_{\epsilon,\overline{\xi}}^{\perp}$
 defined by
 \[
  T_{\ep,\overline{\xi}}(\phi)=L_{\ep,\overline{\xi}}^{-1}(N_{\ep,\overline{\xi}}(\phi)+R_{\ep,\overline{\xi}}).
 \]
By Lemma \ref{L_inversible} we have

\begin{equation}
\label{RN}
 \Vert T_{\ep,\overline{\xi}}(\phi) \Vert_\ep \leq c\Big(\Vert N_{\ep,\overline{\xi}}(\phi)\Vert_\ep +\Vert R_{\ep,\overline{\xi}}\Vert_\ep \Big)
 \end{equation}
and 
 \[
 \Vert T_{\ep,\overline{\xi}}(\phi_1) -  T_{\ep,\overline{\xi}}(\phi_2)\Vert_\ep \leq c\Big(\Vert N_{\ep,\overline{\xi}}(\phi_1)\Vert_\ep - \Vert N_{\ep,\overline{\xi}}(\phi_2)\Vert_\ep \Big).
 \]
 By $(\ref{norma_q})$ and $(\ref{norma_ep})$, it holds
 
\[
 \Vert N_{\ep,\overline{\xi}}(\phi)\Vert_\ep \leq C \mid f(V_{\ep,\overline{\xi}}+\phi)-f(V_{\ep,\overline{\xi}})-f'(V_{\ep,\overline{\xi}})\phi \mid_{p',\ep}.
\]
And by the mean value Theorem, there is some $\tau\in(0,1)$ such that,  if $||\phi_1||_{\epsilon}$ and  $||\phi_2||_{\epsilon}$ are small enough,
 
\begin{equation}
\begin{array}{lcc}
    \mid 
 f(V_{\ep,\overline{\xi}}+\phi_1)-f(V_{\ep,\overline{\xi}}+\phi_2)-f'(V_{\ep,\overline{\xi}})(\phi_1-\phi_2) \mid_{p',\ep} \\
 \leq C \mid 
 (f'(V_{\ep,\overline{\xi}}+\phi_2 + \tau(\phi_1-\phi_2))-f'(V_{\ep,\overline{\xi}}))(\phi_1-\phi_2) \mid_{p',\ep} \\
  \leq C \mid 
 f'(V_{\ep,\overline{\xi}}+\phi_2 + \tau(\phi_1-\phi_2))-f'(V_{\ep,\overline{\xi}})\mid_{\frac{p}{p-2},\ep}\mid \phi_1-\phi_2 \mid_{p',\ep} \\
\end{array}
\end{equation}

It follows from \cite{Dancer}, section 3, that 

\begin{equation}
\begin{array}{lcc}
\mid 
 
f'(V_{\ep,\overline{\xi}}+\phi_2 + \tau(\phi_1-\phi_2))-f'(V_{\ep,\overline{\xi}})\mid_{\frac{p}{p-2},\ep}\mid \phi_1-\phi_2 \mid_{p',\ep} \\
 \leq C \Vert \phi_1 -\phi_2\Vert_\ep.\\
\end{array}
\end{equation}

\noindent And then we have

\begin{equation}
\label{fixed}
\begin{array}{lcc}
\Vert T_{\ep,\overline{\xi}}(\phi_1) -  T_{\ep,\overline{\xi}}(\phi_2)\Vert_\ep 
\leq \Vert N_{\ep,\overline{\xi}}(\phi_1)- N_{\ep,\overline{\xi}}(\phi_2)\Vert_\ep \leq c \Vert \phi_1 -\phi_2\Vert_\ep\\
\end{array}
\end{equation}
for $c \in (0,1)$, provided $||\phi_1||_{\epsilon}$ and  $||\phi_2||_{\epsilon}$ are small enough.

Hence $T_{\epsilon,\bar \xi}$ has a fixed point in a small enough ball in $K^{\perp}_{\epsilon,\bar \xi}$, centered at 0. 

Moreover, for such fixed point, we have  by eq. (\ref{RN}),

$$ ||\phi_{\epsilon,\bar \xi}||_{\epsilon}= \Vert T_{\ep,\overline{\xi}}(\phi) \Vert_\ep \leq c\Big(\Vert N_{\ep,\overline{\xi}}(\phi)\Vert_\ep +\Vert R_{\ep,\overline{\xi}}\Vert_\ep \Big).$$

On the other hand
\begin{equation}
\label{Nphi}
||N_{\ep,\overline{\xi}}(\phi)||_{\epsilon}\leq c||\phi||_{\epsilon},
\end{equation}

for $\phi$ with $||\phi||_{\epsilon}$ small enough, since 

$$||N_{\ep,\overline{\xi}}(\phi)||_{\epsilon}\leq c\left(||\phi||_{\epsilon}^{p-1}+||\phi||_{\epsilon}^2\right),$$

\noindent by eq (3.35) in \cite{Dancer}. 

Hence by Lemma \ref{estimateR}, and inequality (\ref{Nphi}),

$$ ||\phi_{\epsilon,\bar \xi}||_{\epsilon} \leq c\Big(\Vert N_{\ep,\overline{\xi}}(\phi)\Vert_\ep +\Vert R_{\ep,\overline{\xi}}\Vert_\ep \Big)\leq  c_1||\phi_{\epsilon,\bar \xi}||_{\epsilon} +c_2   \Big(\ep^2 + \sum_{i\neq j}e^{-\frac{1+\sigma}{2}\frac{d_g(\xi_i,\xi_j)}{\ep}}\Big)$$

This implies the estimate of the Lemma:

$$ ||\phi_{\epsilon,\bar \xi}||_{\epsilon} \leq c_3   \Big(\ep^2 + \sum_{i\neq j}e^{-\frac{1+\sigma}{2}\frac{d_g(\xi_i,\xi_j)}{\ep}}\Big).$$

Finally, to prove that the map $\xi \rightarrow \phi_{\epsilon,\bar \xi}$ is in fact a $C^1$ map, given $\epsilon$, we use the Implicit Function Theorem applied to the function
$$F(\bar \xi, \phi)= T_{\ep,\overline{\xi}}(\phi)- \phi.$$

As stated above, eq. (\ref{fixed}) guarantees that there is some $\phi_{\epsilon,\bar \xi} $, such that $F(\bar \xi, \phi_{\epsilon,\bar \xi})=0$. Also, $T_{\ep,\overline{\xi}}(\phi)$ is  differentiable, with  differentiable inverse $L_{\ep,\overline{\xi}}(\phi)$. The Implicit Function Theorem then implies that $\xi \rightarrow \phi_{\epsilon,\bar \xi}$ is  a $C^1$ map.

 \hfill$\square$

\section{Appendix}

In this section we compute numerically  $\beta$ of section 3, for low dimensions. Namely
\begin{equation}
\label{betaapp}
\beta:= {\bf c} \displaystyle\int_{\mathbb{R}^n} U^2(z)  \ dz  -  \frac{1}{n(n+2)}\displaystyle\int_{\mathbb{R}^n} |\nabla U(z)|^2  |z|^2 \ dz,
\end{equation}
\noindent which plays an important role in the asymptotic expansion of the energy $\overline{J_{\epsilon}}$.

\begin{table}
	\centering
	\begin{tabular}{|c|c|c|c|c|} 
		\hline
		&&&&\\
		m & n &  $V_{n-1}\ {\bf c} \int_{0}^{\infty} u^2 r^{n-1}dr$     & $\frac{V_{n-1}}{n(n+2)}\int_{0}^{\infty} u'^2 r^{n+1} dr $   & $\beta$  \\[5pt]
		\hline
		\hline
		2 & 2 & 1.9502 & 2.331  & -0.38089 \\
		2 & 3 & 11.959 & 13.259 & -1.2999  \\
		2 & 4 & 81.771 & 87.5   & -5.7285  \\
		2 & 5 & 617.47 & 647.82 & -30.353  \\
		2 & 6 & 5083.3 & 5268.8 & -185.5   \\
		2 & 7 & 45119  & 46391  & -1272.4  \\
		3 & 2 & 3.9303 & 4.4149 & -0.48461 \\
		3 & 3 & 26.196 & 28.329 & -2.1329  \\
		3 & 4 & 194.26 & 205.59 & -11.324  \\
		3 & 5 & 1577.6 & 1647.1 & -69.453  \\
		3 & 6 & 13854  & 14332  & -478.38  \\
		4 & 2 & 6.2006 & 6.7579 & -0.55731 \\
		4 & 3 & 45.28  & 48.231 & -2.9513  \\
		4 & 4 & 363.46 & 381.54 & -18.085  \\
		4 & 5 & 3162.7 & 3287.2 & -124.58  \\
		5 & 2 & 8.6442 & 9.2554 & -0.61113 \\
		5 & 3 & 68.674 & 72.419 & -3.7455  \\
		5 & 4 & 592.7  & 618.4  & -25.692  \\
		6 & 2 & 11.199 & 11.851 & -0.65243 \\
		6 & 3 & 95.938 & 100.42 & -4.4788     \\
		\hline      
	\end{tabular}
	\vspace{5pt}
	\caption{Numerical values for $\beta$, for $n+m \leq9$.}
	\label{numericalbetas}
\end{table}

$\beta$ is a dimensional constant that requires knowledge  of  the unique (up to translations) positive solution  $U \in H^1(\mathbb{R}^n)$ that vanishes at infinity of 
\begin{equation}
-\Delta U + U=U^{p-1} \hbox{ in } \mathbb{R}^n,
\end{equation}
\noindent with $p=\frac{2(m+n)}{m+n-2}$. The solution is known to exist, and to be unique and radial, see \cite{Kwong} for details.

Hence, we consider the solution $h=h_{\alpha}$ of

\begin{equation}
\label{hache}
h''(t)+\frac{n-1}{t}h'(t)-h(t)+h(t)^{\frac{m+n+2}{m+n-2}}=0.
\end{equation}

\noindent with $h(0)=\alpha>0$, $h'(0)=0$.
By the aforementioned existence and uniqueness results, there exists only one value
$\alpha=\alpha_0=\alpha_0(m,n)$ that gives a positive solution $h_{\alpha_0}$ that vanishes at infinity. Our approach is to find  $h_{\alpha_0}$ numerically as the solution of (\ref{hache}) that vanishes at infinity, and then to integrate it numerically to find 
$V_{n-1} \ {\bf c}  \int_{0}^{\infty} u^2 r^{n-1}dr$    and $\frac{V_{n-1}}{n(n+2)} \int_{0}^{\infty} u'^2 r^{n+1} dr $, the two terms involved in (\ref{betaapp}).
Of course $u(r)=h_{\alpha_0}(r)$.

In Table \ref{numericalbetas} we show the  numerical results, where $\beta$ is negative for  $m+n\leq9$.


\begin{thebibliography}{88}


\bibitem{Brendle} Brendle, {\it Blow-up phenomena for the Yamabe problem}, J. Am. Math. Soc. {\bf 15} (2008), 951-979.


\bibitem{Akutagawa} K. Akutagawa, L. Florit, J. Petean, {\it On Yamabe constants of Riemannian
products}, Comm. Anal. Geom. {\bf 15} (2007), 947-969. 


\bibitem{Aubin} T. Aubin, {\it Equations differentielles non-lineaires et
probleme de Yamabe concernant la courbure scalaire},
J. Math. Pures Appl. {\bf 55} (1976), 269-296.


\bibitem{BP1} R. Bettiol, P. Piccione, {\it 
	Bifurcation and local rigidity of homogeneous solutions to the Yamabe problem on spheres},  
	Calc. Var. Partial Differential Equations 
	\textbf{47},  789--807 (2013). 


\bibitem{BP2} R. Bettiol, P. Piccione, {\it 
	Multiplicity of solutions to the Yamabe problem on collapsing Riemannian submersions}, 
	Pacific J. Math. 
	\textbf{266}, 1--21 (2013).


\bibitem{Piccione} L.L. de Lima, P. Piccione, M. Zedda, {\it On bifurcation of solutions of the Yamabe problem in
product manifolds}, Annales de L'institute Henri Poincare (C) Non Linear Analysis {\bf 29}, 261-277 (2012).



\bibitem{Dancer} 
\newblock{E. N. Dancer, A. M. Micheletti, A. Pistoia}, \newblock{\em Multipeak solutions for some singularly perturbed nonlinear elliptic problems on Riemannian manifolds}, 
Manuscripta Math. 128 (2009), 163-193.


\bibitem{DKM} 
\newblock{S. Deng, Z. Khemiri and F. Mahmoudi}, \newblock{\em On spike solutions for a singularly perturbed problem in a compact Riemannian manifold}. Commun. Pure Appl. Anal. 17 (2018), no. 5, 2063–2084. 



\bibitem{Floer} A. Floer,  A. Weinstein, {\it Nonspreading wave packets for the cubic Schrodinger equation with a bounded potential}, J. Funct. Anal. {\bf 69}, 397-408 (1986). 


\bibitem{Trudinger} D.Gilbarg, N. Trudinger,
{\it Elliptic Partial Differential Equations of Second Order}, Springer-Verlag, Berlin, 1983.


\bibitem{G} 
\newblock{A. Gray}, \newblock{\em The volume of a small geodesic ball of a Riemannian manifold}. Michigan Math. J.
Volume 20, Issue 4 (1974), 329-344.



\bibitem{Henry} G. Henry, J. Petean, {\it Isoparametric hypersurfaces and metrics of constant scalar curvature}, Asian J. Math. {\bf 18}, 53-68, (2014). 



\bibitem{Kwong} M. K. Kwong, {\it Uniqueness os positive solutions of $\Delta u -u +u^p =0$ in $\re^n$}, Arch. Rational
Mech. Anal. {\bf 105} (1989), 243-266. 

\bibitem{Li} Y. Li, {\it On a singularly perturbed equation with Neumann boundary condition}, Communications in Partial Differential Equations {\bf 23}, 487-545, (1998).

\bibitem{Micheletti} 
\newblock{A. M. Micheletti, A. Pistoia}, \newblock{\em The role of the scalar curvature in a nonlinear elliptic problem on Riemannian manifolds}, Calc. Var. 34 (2009), 233-265.

\bibitem{Micheletti2} \newblock{A. M. Micheletti, A. Pistoia}, \newblock{\em Generic properties of critical points of the scalar curvature  on 
a Riemannian manifold}, Proc. Amer. Math.Soc. 138 (2010), 3277-3284.



\bibitem{Otoba}N. Otoba, J. Petean, {\it Solutions of the Yamabe equation on harmonic Riemannian submersions}, {\bf arXiv:1611.06709}.

\bibitem{Parker} T. H. Parker and J. M.  Lee, {\it The Yamabe Problem} Bull. of the Amer. Math. Soc. {\bf 17}, Number 1, (1987), 37-91.

\bibitem{Petean} J. Petean, {\it Multiplicity results for the Yamabe equation by Lusternik-Schnirelmann theory}, {\bf arXiv:1611.01177}.


\bibitem{Petean2} J. Petean,  
	{\it Metrics of constant scalar curvature conformal to Riemannian products},   
	Proc. Amer. Math. Soc. 
	{\bf 138}, 2897--2905, (2010).



\bibitem{Pollack} D. Pollack, {\it Nonuniqueness and high energy solutions for a conformally invariant scalar
equation}, Comm. Anal. Geom. {\bf 1}, 347-414 (1993).

\bibitem{Schoen} R. Schoen, {\it Conformal deformation of a Riemannian metric to constant scalar curvature}, J. Differential Geometry {\bf 20} (1984), 479-495.


 \bibitem{Schoen2} R. Schoen, {\it Variational theory for the total scalar curvature functional for Riemannian metrics and related topics}, Lecture Notes in Math. 1365,
Springer-Verlag, Berlin, 1989, 120-154.





\bibitem{Trudinger2}  N. Trudinger, {\it Remarks concerning the conformal deformation of Riemannian structures on compact manifolds}, Ann. Scuola Norm. Sup. Pisa (3) {\bf 22} (1968), 265–274.

\bibitem{Yamabe} H. Yamabe, {\it On a deformation of Riemannian 
structures on compact manifolds}, Osaka Math. J. {\bf 12} (1960),
21-37.




\end{thebibliography}
\end{document}